\documentclass[hidelinks,onefignum,onetabnum]{siamart250211}

\usepackage{lipsum}
\usepackage{amsfonts}
\usepackage{graphicx}
\usepackage{epstopdf}
\usepackage{algorithm}
\usepackage{algpseudocode}

\ifpdf
  \DeclareGraphicsExtensions{.eps,.pdf,.png,.jpg}
\else
  \DeclareGraphicsExtensions{.eps}
\fi
\usepackage{amsmath}
\usepackage{amssymb}
\usepackage{tikz}
\usetikzlibrary{positioning, arrows.meta}

\usepackage{subcaption}
\usepackage{amsopn}
\usepackage{enumitem}
\usepackage{mathtools}

\newsiamremark{remark}{Remark}
\newsiamremark{hypothesis}{Hypothesis}
\crefname{hypothesis}{Hypothesis}{Hypotheses}
\newsiamthm{claim}{Claim}
\newsiamremark{fact}{Fact}
\crefname{fact}{Fact}{Facts}

\headers{Refinement-based Christoffel sampling (RCS)}{A. Herremans and B. Adcock}

\title{Refinement-based Christoffel sampling for least squares approximation in non-orthogonal bases\thanks{Submitted to the editors DATE.}}

\author{Astrid Herremans\thanks{Department of Computer Science, KU Leuven, 3001 Leuven, Belgium (\email{astrid.herremans@kuleuven.be}).}
\and Ben Adcock\thanks{Department of Mathematics, Simon Fraser University, Burnaby, BC V5A 1S6, Canada
  (\email{ben\_adcock@sfu.ca}, \url{http://www.benadcock.org}).}}

\DeclareMathOperator*{\esssup}{ess\,sup}
\DeclareMathOperator*{\essinf}{ess\,inf}
\DeclareMathOperator{\SPAN}{span}
\DeclareMathOperator*{\argmin}{arg\,min}
\newcommand{\norm}[1]{\left\lVert#1\right\rVert}

\DeclareMathOperator{\Tr}{Tr}

\ifpdf
\hypersetup{
  pdftitle={},
  pdfauthor={}
}
\fi

\begin{document}

\maketitle

\begin{abstract}
We introduce a refinement-based Christoffel sampling (RCS) algorithm for least squares approximation in the span of a given, generally non-orthogonal set of functions $\Phi_n = \{\phi_1, \dots, \phi_n\}$. A standard sampling strategy for this problem is Christoffel sampling, which achieves near-best approximations in probability using only $\mathcal{O}(n \log(n))$ samples. However, it requires i.i.d.\ sampling from a distribution whose density is proportional to the inverse Christoffel function $k_n$, the computation of which requires an orthonormal basis. As a result, existing approaches for non-orthogonal bases $\Phi_n$ typically rely on costly discrete orthogonalization. We propose a new iterative algorithm, inspired by recent advances in approximate leverage score sampling, that avoids this bottleneck. Crucially, while the computational and memory cost of discrete orthogonalization scales with $\|k_n\|_{L^\infty(X)}$, the cost of our approach increases only logarithmically in $\|k_n\|_{L^\infty(X)}$. In addition, we account for finite-precision effects by considering a numerical variant of the Christoffel function, ensuring that the algorithm relies only on computable quantities. Alongside a convergence proof, we present extensive numerical experiments demonstrating the efficiency and robustness of the proposed method.
\end{abstract}



\section{Introduction}
We consider the problem of approximating a function $f \in L^2_\rho(X)$, defined on $X \subseteq \mathbb{R}^d$, using weighted least squares fitting in the span of a given set of functions $\Phi_n = \{\phi_1, \dots, \phi_n\}$
\begin{equation} \label{eq:weightedLS}
    \tilde{f} \in \argmin_{v \in \SPAN(\Phi_n)} \sum_{k=1}^m w_k \left\vert f(x_k) - v(x_k) \right\vert^2,
\end{equation}
where the set of sample points $\{x_k\}_{k=1}^m$ and the weights $\{w_k\}_{k=1}^m$ can be chosen freely. The objective is to attain near-best approximation errors while using a small number of sample points.

A common strategy is to use sample points that are drawn i.i.d.\ from some sampling distribution, allowing for an analysis using tools from random matrix theory. This led to the sampling distribution $\mu_n$, first proposed for (total degree) polynomial spaces~\cite{hampton2015coherence} and later generalized~\cite{cohen2017optimal}, which achieves a near-optimal approximation error in probability using only $m = \mathcal{O}(n \log n)$ samples. This is close to the optimal case of interpolation, which corresponds to $m=n$\footnote{In the introduction, we assume for simplicity of exposition that $\dim\left(\SPAN\left(\{\phi_i\}_{i=1}^n\right)\right) = n$,  i.e., that $\Phi_n$ forms a basis. This assumption is not required and is relaxed later.}. The sampling distribution $\mu_n$ is defined by
\[
    w_n d\mu_n = d\rho \qquad \text{with} \qquad w_n = \frac{n}{\sum_{i=1}^n \lvert u_i \rvert^2},
\]
where $\{u_i\}_{i=1}^n$ is an orthonormal basis for $\SPAN(\Phi_n)$. The function 
\[
    k_n = \sum_{i=1}^n \lvert u_i \rvert^2
\]
is known as the inverse of the \textit{Christoffel function}~\cite{nevai1986geza} and, hence, this sampling strategy is often referred to as \textit{Christoffel sampling}~\cite{adcock2024optimal}.

An important practical challenge is that sampling from $\mu_n$ requires evaluating the Christoffel function, which in turn typically requires access to an orthonormal basis. We are particularly interested in settings where such a basis is not available, i.e., when the given basis $\Phi_n$ is non-orthogonal. Non-orthogonal bases often arise when one incorporates prior knowledge about the target functions, such as known singularities, oscillations or other features that a standard polynomial basis would fail to capture efficiently. They also naturally appear when working on irregular domains. In addition, non-orthogonal bases can result from adaptive or problem-driven constructions, such as bases learned using machine learning techniques.

In earlier work~\cite{dolbeault2022optimal,adcock2020near,migliorati2021multivariate}, this issue was addressed by discretely orthogonalizing the basis $\Phi_n$ on a dense grid. More specifically, we can equivalently characterize the inverse Christoffel function $k_n$ using the non-orthogonal basis $\Phi_n$ by
\begin{equation} \label{eq:christfun}
    k_n(x) = \phi(x)^* G^{-1} \phi(x),
\end{equation}
where $\phi(x) = \begin{bmatrix} \phi_1(x) & \dots & \phi_n(x) \end{bmatrix}^\top$ and $G$ is the continuous Gram matrix associated with $\Phi_n$, defined by
\[
    (G)_{ij} = \langle \phi_i, \phi_j \rangle_{L^2_\rho(X)}.
\]   
Discrete orthogonalization amounts to replacing $G$ by a discrete Gram matrix $G_\ell$, such as
\[
    (G_\ell)_{ij} = \sum_{k=1}^\ell \frac{1}{\ell} \phi_i(t_k) \overline{\phi_j(t_k)},
\]
where the points $\{t_k\}_{k=1}^\ell$ are sampled i.i.d.\ from the measure $\rho$. Following~\cite{dolbeault2022optimal,adcock2024optimal}, one needs $\ell = \mathcal{O}(\|k_n\|_{L^\infty(X)} \log(n))$ for $G_\ell$ to approximate $G$ well with high probability. Although this is typically an offline cost---computed once for every $\Phi_n$ and independently of the function to be approximated---it can still be prohibitively expensive, especially when the Christoffel function $k_n$ is sharply peaked. Notably, this is precisely the regime where Christoffel sampling offers significant advantages over standard least squares fitting using samples drawn from $\rho$. When $\ell$ is too small, the efficiency of the resulting Christoffel sampling procedure is limited~\cite{trunschke2024optimal}.

We adopt a different strategy inspired by~\cite{cohen2015uniform}; Section~\ref{sec:discretesetting} explains how this discrete method relates to our continuous setting. We aim to construct a function $u$ that satisfies $u(x) \geq k_n(x), \forall x \in X$ while $\|u\|_{L^1_\rho(X)} \leq Cn$ for some small constant $C \geq 1$. When such a function is known, it suffices to draw $m = \mathcal{O}(\|u\|_{L^1_\rho(X)} \log(n)) = \mathcal{O}(n \log(n))$ samples from the sampling distribution $\mu_u$ defined by 
\[
    w_u d\mu_u = d\rho \qquad \text{with} \qquad w_u = \frac{\|u\|_{L^1_\rho(X)}}{u}
\]
in order to obtain near-optimal weighted least squares approximations in probability. We construct such a function $u$ efficiently using an iterative method in which, at each step, we sample points from the current $u$ and thereafter use these samples to update and improve $u$. 

In particular, we show that $\|u\|_{L^1_\rho(X)}$ decreases by a constant factor in each iteration until it is close to $\|k_n\|_{L^1_\rho(X)} = n$, using only $\mathcal{O}(n \log(n))$ samples per step. Starting from the initialization $u(x) = \|k_n\|_{L^\infty(X)}$, it follows that the total number of samples required to obtain an adequate $u$ is of the order $\mathcal{O}(\log(\|k_n\|_{L^\infty(X)}) n \log(n))$. Therefore, in contrast to the method described above, the computational complexity of this approach depends only logarithmically on $\|k_n\|_{L^\infty(X)}$, rather than linearly. Inspired by~\cite[Algorithm 3]{cohen2015uniform}, the algorithm is termed \textit{refinement-based Christoffel sampling} (RCS).

\subsection{Effects of finite-precision arithmetic}
Another issue related to least squares fitting with non-orthogonal bases is the influence of numerical rounding errors. Even though every basis is linearly independent from an analytical perspective, it can be indistinguishable from a linearly dependent set in finite-precision arithmetic. Numerically computed least squares fits using such a basis behave either explicitly or implicitly as $\ell^2$-regularized approximations with a regularization parameter $\epsilon \propto \epsilon_\text{mach} \sqrt{\|G\|_2}$ due to rounding errors~\cite{adcock2019frames,adcock2020frames,herremans2025sampling}. Whereas this numerical regularization has a negligible influence if the basis is (nearly) orthogonal, it becomes significant when $\kappa(G) \geq 1/\epsilon_\text{mach}^2$, in which case $\Phi_n$ is called numerically redundant~\cite{herremans2025sampling}. 

Using similar tools from random matrix theory, it was found in~\cite{herremans2025sampling} that the sampling distribution $\mu_n^\epsilon$ defined by
\[
    w_n^\epsilon d\mu_n^\epsilon = d\rho \qquad \text{with} \qquad w_n^\epsilon = \frac{n^\epsilon}{k_n^\epsilon}
\]
is of interest in this case. Here,
\[
    k_n^\epsilon(x) = \phi(x)^* (G+\epsilon^2 I)^{-1} \phi(x) \qquad \text{and} \qquad n^\epsilon = \sum_{i=1}^n \frac{\lambda_i(G)}{\lambda_i(G) + \epsilon^2} = \|k_n^\epsilon\|_{L^1_\rho(X)}
\]
are called the \textit{numerical inverse Christoffel function} and the \textit{numerical dimension}, respectively, and $\lambda_i(G)$ denotes the $i$-th eigenvalue of $G$. The numerical dimension $n^\epsilon$ now plays the role of the original dimension $n$, so that only $\mathcal{O}(n^\epsilon \log(n^\epsilon))$ samples are required for accurate least squares fitting, compared to $\mathcal{O}(n \log(n))$ when rounding errors are not accounted for. Note that $n^\epsilon \leq n$, reflecting the principle that numerical rounding errors can only reduce the dimension of the approximation space.

Even though in many practical scenarios one has $n^\epsilon \propto n$, such that accounting for the influence of rounding errors does not have a significant impact on the sampling complexity, the algorithm is based on the numerical inverse Christoffel function. A few observations clarify why this choice is necessary. First, it ensures that our algorithm only works with computable quantities. By constrast, evaluating the standard Christoffel function~\eqref{eq:christfun} may require inverting an arbitrarily ill-conditioned matrix, which is not feasible in finite precision. Moreover, observe that
\[
    G \preceq G + \epsilon^2 I \preceq \left( 1 + \frac{\epsilon^2}{\lambda_n(G)} \right) G \quad \Rightarrow \quad \left( 1 + \frac{\epsilon^2}{\lambda_n(G)} \right)^{-1} k_n \leq k_n^\epsilon \leq k_n,
\]
and, similarly for the numerical dimension $n^\epsilon$,
\[
   \frac{\lambda_n(G)}{\lambda_n(G) + \epsilon^2} \leq \frac{\lambda_i(G)}{\lambda_i(G) + \epsilon^2} \leq 1 \quad \Rightarrow \quad \left( 1 + \frac{\epsilon^2}{\lambda_n(G)} \right)^{-1} n \leq n^\epsilon \leq n.
\]
Thus, $k_n^\epsilon$ and $n^\epsilon$ deviate from $k_n$ and $n$, respectively, by more than a factor of two only when $\lambda_n(G) \leq \epsilon^2 \propto \lambda_1(G) \epsilon_{\text{mach}}^2$, i.e., when $\Phi_n$ is (nearly) numerically redundant. Conversely, when $\Phi_n$ is close to orthogonal, they behave almost identically to $k_n$ and $n$. This shows that the algorithm remains meaningful even when rounding errors are negligible. Finally, the numerical regularization parameter $\epsilon$ plays a crucial role in the convergence proof of the algorithm: the proof of Lemma~\ref{lm:weighting} heavily relies on the assumption that $\epsilon > 0$.

\subsection{Connection to the discrete setting} \label{sec:discretesetting}
Christoffel sampling is widely recognized as being closely related to leverage score sampling, a well-established technique in numerical linear algebra~\cite{drineas2006sampling, adcock2024optimal}. Given an $m \times n$ matrix $A$ with $m > n$, the $j$-th statistical leverage score $\ell_j$ is defined by
\[
    \ell_j = \sum_{i=1}^n \lvert (U)_{ji}\rvert^2 \qquad 1 \leq j \leq m,
\]
where $U$ denotes the $m \times n$ matrix containing the left singular vectors of the matrix $A$~\cite{drineas2012fast}. Leverage scores identify which rows of a matrix are most important and are typically used to subsample a tall-and-skinny matrix $A$ before solving a least squares problem $Ax \approx b$. Other applications include, but are not limited to, low-rank approximation, matrix completion and kernel-based machine learning~\cite{drineas2016randnla}. 

From the definition of $\ell_j$, it is clear that leverage scores are a discrete analogue of the inverse Christoffel function $k_n$. In particular, the continuous variable $x \in X$ is substituted by a discrete index $j$, where $1 \leq j \leq m$. Furthermore, the Christoffel function serves a similar role in approximation theory as the leverage scores in numerical linear algebra, namely in subsampling the $\infty$-dimensional regression problem 
\[
    \begin{bmatrix} \phi_1 & \phi_2 & \dots & \phi_n \end{bmatrix} c \approx f.
\]
The numerical Christoffel function $k_n^\epsilon$ can similarly be seen as a continuous analogue of ridge leverage scores~\cite{herremans2025sampling}, which arise in the context of $\ell^2$-regularized least squares problems~\cite{alaoui2015fast}.

Challenges similar to those encountered in Christoffel sampling also appear in the context of leverage scores. In particular, computing leverage scores requires the full singular value decomposition of $A$, which is as computationally expensive as solving the original least squares problem. Consequently, there has been substantial research on fast computation---or rather,  approximation---of leverage scores. One approach introduced in~\cite{cohen2015uniform,cohen2017input} consists of computing a tight upper bound to the (rigde) leverage scores via an efficient iterative scheme. This approach can be extended to the continuous setting and forms the foundation of this paper.

\subsection{Other related work}
Other methods have been proposed for efficiently computing the Christoffel function from a non-orthogonal basis. For example, in ~\cite[\S4]{dolbeault2022optimal} a multilevel strategy is introduced that approximates the inverse Christoffel function through successive refinements based on nested subspaces. However, this approach relies on rather strong assumptions about these subspaces and does not account for the effects of finite-precision arithmetic. Another iterative method is presented in~\cite{trunschke2024optimal}, which requires the strong assumption $\lambda_{\min}(G) \gg \epsilon_{\text{mach}}$---where $\lambda_{\min}$ indicates the smallest nonzero eigenvalue---for computability. 

Moreover, we note that optimizations have been studied to improve upon the Christoffel sampling scheme. These strategies reduce the number of required samples from $\mathcal{O}(n\log(n))$ to $\mathcal{O}(n)$, typically by leaving the i.i.d.\ setting. For instance, in~\cite{bartel2023constructive} a subsampling scheme is introduced that is able to compress a good sampling set, while the algorithm in~\cite{dolbeault2024randomized} is designed to directly sample from a distribution that incorporates both the Christoffel function and effective resistance. Implementing the latter can be done via rejection sampling with draws from the Christoffel function, as mentioned in~\cite[\S5]{dolbeault2024randomized}, which makes our approach directly applicable in this setting as well.

The link between leverage scores and the Christoffel function is not new; a more extensive comparison can be found in~\cite[\S2]{meyer2023near}. Furthermore, in the discrete setting many more algorithms have been proposed to efficiently approximate leverage scores, such as~\cite{drineas2012fast} which uses random sketching to speed up the computations. Viewing the object $\begin{bmatrix} \phi_1 & \phi_2 & \dots & \phi_n \end{bmatrix}$ as the limit of a tall-and-skinny matrix was popularized in the Matlab package chebfun, where it is referred to as a quasimatrix~\cite{driscoll2014chebfun}.

\subsection{Overview of the paper}
The remainder of the paper is organized as follows. In Section~\ref{sec:symbols}, the symbols and notation used throughout the paper are summarized. In Section~\ref{sec:algorithm}, we introduce the proposed iterative algorithm, outline its main objective and provide an initial discussion of its computational performance. Section~\ref{sec:theory} presents a theoretical analysis of the method, including a proof of its convergence properties. Section~\ref{sec:numericalexperiments} is devoted to extensive numerical experiments that illustrate the efficiency and robustness of our approach. Finally, Section~\ref{sec:conclusion} lists some conclusions and open problems.

\section{Symbols and notation} \label{sec:symbols}
Here, we collect the main symbols used throughout the paper.
Let $X \subseteq \mathbb{R}^d$ denote the domain of interest, equipped with a probability measure $\rho$. The inner product on $L^2_\rho(X)$ is defined by
\[
    \langle f,g \rangle_{L^2_\rho(X)} = \int_X f \;\overline{g} \;d\rho.
\]
We introduce the following sampling measure based on a strictly positive function $u$:
\begin{equation} 
    w_u d\mu_u = d\rho \qquad \text{with} \qquad w_u = \frac{\|u\|_{L^1_\rho(X)}}{u}. \label{eq:usampler}
\end{equation}
Note that this corresponds to the choices
\[
    x_k \overset{\mathrm{iid}}{\sim} \mu_u \qquad \text{and} \qquad w_k = w_u(x_k)/m
\]
in~\eqref{eq:weightedLS}. $\Phi_n = \{\phi_1, \dots, \phi_n\} \subset L^2_\rho(X)$ is a given set of non-orthogonal functions, which may even be linearly dependent. The vector of evaluations of $\Phi_n$ at a point $x \in X$ is denoted by $\phi(x) = \begin{bmatrix}\phi_1(x) & \dots & \phi_n(x)\end{bmatrix}^\top \in \mathbb{C}^n$. The continuous Gram matrix associated with $\Phi_n$ is an $n \times n$ matrix $G$ defined by 
\[
    (G)_{ij} = \langle \phi_i, \phi_j \rangle_{L^2_\rho(X)}.
\]
The numerical inverse Christoffel function equals  
\[
    k_n^\epsilon(x) = \phi(x)^* (G + \epsilon^2 I)^{-1} \phi(x)
\]
where $\epsilon > 0$ models finite-precision rounding and $^*$ denotes the Hermitian conjugate. The numerical dimension is defined by 
\[
    n^\epsilon = \sum_{i=1}^n \frac{\lambda_i(G)}{\lambda_i(G) + \epsilon^2},
\]
where ${\lambda_i(G)}$ are the eigenvalues of $G$. For Hermitian matrices $A$ and $B$, we use $A \preceq B$ to denote the Loewner order, i.e., $B-A$ is positive semidefinite. 

\section{Description of the algorithm}\label{sec:algorithm}
\subsection{Main objective}
The goal is to identify a small number of samples $\{x_k\}_{k=1}^m$ and weights $\{w_k\}_{k=1}^m$ so that the discrete least squares fit to a function $f$ in a given set $\Phi_n$ has near-optimal approximation error. Our approach to obtain such samples is to compute a function $u$ that satisfies 
\begin{equation} 
    u(x) \geq k_n^\epsilon(x), \; \forall x \in X \qquad \text{while} \qquad \|u\|_{L^1_\rho(X)} \leq Cn^\epsilon  \label{eq:urequirements}
\end{equation}
for some small constant $C \geq 1$ and numerical parameter $\epsilon \propto  \epsilon_\text{mach}\sqrt{\|G\|_2}$. In this section, we verify that samples drawn from the associated measure $\mu_u$~\eqref{eq:usampler} indeed result in stable and accurate least squares fits, both when $\Phi_n$ is numerically redundant and when numerical rounding errors are negligble, i.e., if $\Phi_n$ is not too far from being orthogonal.

Most importantly, if the sample points $\{x_k\}_{k=1}^m$ are drawn independently from $\mu_u$, then, with high probability, the discrete Gram matrix $G_m$ defined by
\[
    (G_m)_{ij} = \sum_{k=1}^m \frac{w_u(x_k)}{m} \phi_i(x_k) \overline{\phi_j(x_k)},
\]
closely approximates the continuous Gram matrix $G$ using only $m = \mathcal{O}(n^\epsilon \log(n^\epsilon))$ samples. This is formalized in Lemma~\ref{lm:matrixineq}, which generalizes~\cite[Theorem 10]{herremans2025sampling}.

\begin{lemma} \label{lm:matrixineq}
    Let $\Phi_n = \{\phi_i\}_{i=1}^n \subset L^2_\rho(X)$ be a set of functions, where $X \subseteq \mathbb{R}^d$, $\epsilon \geq 0$, and $u$ be an integrable function satisfying 
    \[
        u(x) \geq k_n^\epsilon(x) \qquad \forall x \in X.
    \]
    Assuming that $k_n^\epsilon(x) > 0, \forall x \in X$, let $\Delta > 0$ and $\gamma < 1$ and generate
    \[
        m \geq c_\Delta \log(16 n^\epsilon/\gamma) \|u\|_{L^1_\rho(X)}
    \]
    samples $\{x_k\}_{k=1}^m$ i.i.d.\ from $\mu_u$~\eqref{eq:usampler}, where $c_\Delta = \frac{1+2\Delta/3}{\Delta^2/2}$. Then, assuming $\|G\|_2 \geq \epsilon^2$, one has
    \begin{equation} \label{eq:matrixineq}
        (1-\Delta)(G+\epsilon^2 I) \preceq G_m + \epsilon^2 I \preceq (1+\Delta) (G + \epsilon^2 I)
    \end{equation}
    with probability at least $1-\gamma$.
\end{lemma}
\begin{proof}
    This follows from a straightforward generalization of the proof of~\cite[Theorem 10]{herremans2025sampling} using a generic $\Delta$ instead of $\Delta = 1/2$, and from
    \[
        \|w_u k^\epsilon_n\|_{L^\infty(X)} = \esssup_{x \sim \rho} \frac{k_n^\epsilon(x) \|u\|_{L^1_\rho(X)}}{u(x)} \leq \|u\|_{L^1_\rho(X)}.
    \]
\end{proof}

\begin{remark}
    Throughout this paper we will assume that $k_n^\epsilon(x) > 0, \forall x \in X$. This is a standard assumption in Christoffel sampling schemes and is satisfied if at every point $x \in X$ at least one function $\phi_i$ is nonzero. This condition is mild and, for example, holds whenever the constant function is included. An alternative is to slightly modify the Christoffel density, as explained in~\cite[\S6.1]{adcock2024optimal}.
\end{remark}

Using~\cite[Theorem 2]{herremans2025sampling}, this immediately implies that numerically computed least squares fits are stable and accurate.

\begin{corollary} \label{corr1}
    Under the same conditions as Lemma~\ref{lm:matrixineq}, given a function $f \in L^2_\rho(X)$ one has
    \[
        \norm{f-\sum_{i=1}^n \hat{c}_i \phi_i}_{L^2_\rho(X)} \leq \left(1 + \frac{1+\sqrt{2}}{\sqrt{1-\Delta}} \right) \inf_{c \in \mathbb{C}^n} \norm{ f - \sum_{i=1}^n c_i \phi_i}_{L^\infty(X)} + \epsilon \|c\|_2
    \]
     with probability at least $1-\gamma$, where 
     \[
        \hat{c} \in \argmin_{c \in \mathbb{C}^n} \sum_{k=1}^m \frac{w_u(x_k)}{m} \left\vert f(x_k) - \sum_{i=1}^n c_i \phi_i(x_k) \right\vert^2 + \epsilon^2 \|c\|_2^2.
     \]
\end{corollary}

At first glance these results seem tailored to the case where $\Phi_n$ is numerically redundant, yet it is important to note that this is not the case. As mentioned in the introduction, the influence of the numerical regularization parameter $\epsilon$ is only significant when $\Phi_n$ is numerically redundant. Hence, the results from Lemma~\ref{lm:matrixineq} are also useful in the more standard setting where numerical rounding errors are negligible. In particular,~\eqref{eq:matrixineq} implies that
\begin{equation} \label{eq:matrixineq2}
    \left( 1 - \widetilde{\Delta} \right) G \preceq G_m \preceq \left( 1 + \widetilde{\Delta} \right) G \qquad \text{with} \qquad \widetilde{\Delta} = \Delta \left( 1 + \frac{\epsilon^2}{\lambda_1(G)}\right),
\end{equation}
where the term $\epsilon^2 / \lambda_1(G)$ is negligble when $\Phi_n$ is close to orthogonal. From here, it follows that the discrete least squares error is near-optimal using $\mathcal{O}(n^\epsilon \log(n^\epsilon)) = \mathcal{O}(n \log(n))$ samples, by arguments analogous to those used in standard results such as~\cite[Theorem 2.1]{cohen2017optimal}.

\begin{corollary} \label{corr2}
    Under the same conditions as Lemma~\ref{lm:matrixineq}, given a function $f \in L^2_\rho(X)$ one has
    \[
        \norm{f-\tilde{f}}_{L^2_\rho(X)} \leq \left(1 + \frac{1}{\sqrt{1-\widetilde{\Delta}}} \right) \inf_{v \in \SPAN(\Phi_n)} \norm{ f - v}_{L^\infty(X)}
    \]
     with probability at least $1-\gamma$, where $\widetilde{\Delta}$ as in~\eqref{eq:matrixineq2} and
     \[
        \tilde{f} \in \argmin_{v \in \SPAN(\Phi_n)} \sum_{k=1}^m \frac{w_u(x_k)}{m} \left\vert f(x_k) - v(x_k) \right\vert^2.
     \]
\end{corollary}

\subsection{Iterative scheme}
The core strategy for constructing a function $u$ that satisfies~\eqref{eq:urequirements} is an adaptation of the (ridge) leverage score sampling algorithm introduced in~\cite{cohen2015uniform,cohen2017input} to the continuous setting. Suppose we have a candidate function $u$ that upper bounds $k_n^\epsilon$ but with $\|u\|_{L^1_\rho(X)} \gg n^\epsilon$. Drawing fewer than $\mathcal{O}(\|u\|_{L^1_\rho(X)} \log(n^\epsilon))$ samples from the associated measure $\mu_u$~\eqref{eq:usampler} is generally insufficient for accurate least squares fitting. However, these samples can still be used to improve $u$---that is, to construct a new function $u^{(new)}$ that still upper bounds $k_n^\epsilon$ but has a reduced $L_\rho^1(X)$-norm. This procedure is applied iteratively, producing successively tighter upper bounds to $k_n^\epsilon$ until the requirement~\eqref{eq:urequirements} is met. 

In each iteration, we update the function $u$ as follows:
\begin{equation} \label{eq:unew}
    u^{(new)}(x) = \min \left(u(x), \; (1+\Delta) \phi(x)^* (A^*A + \epsilon^2 I)^{-1} \phi(x)\right).
\end{equation}
Here, $A$ consists of $\lceil \alpha m + 1 \rceil$ rows of the $m \times n$ matrix $A_m$ defined by
\[
    (A_m)_{kj} = \sqrt{\frac{w_u(x_k)}{m}} \overline{\phi_j(x_k)} = \frac{\overline{\phi_j(x_k)}}{\sqrt{C_1 u(x_k) \log(n^\epsilon)}} \qquad \text{where} \qquad x_k \overset{\mathrm{iid}}{\sim} \mu_u
\]
with
\[
    m = C_1 \|u\|_{L^1_\rho(X)} \log(n^\epsilon)
\]
for some parameters $\Delta, \alpha \in (0,1]$ and $C_1 \geq 1$. Importantly, since each row of $A_m$ is independent, constructing $A$ only requires drawing $\lceil \alpha m + 1 \rceil$ samples. In other words: $A_m$ is never formed explicitly.

Figure~\ref{fig:schematic} illustrates how this construction ensures that $u^{(new)}$ remains an upper bound to $k_n^\epsilon$ with high probability. According to Lemma~\ref{lm:matrixineq}, for sufficiently large $C_1$,
\[ 
    A_m^*A_m + \epsilon^2 I \preceq (1+\Delta)(G + \epsilon^2 I)
\]
with high probability. Since $A$ is a submatrix of $A_m$, it follows that
\[ 
    A^*A + \epsilon^2 I \preceq A_m^*A_m + \epsilon^2 I \preceq (1+\Delta)(G + \epsilon^2 I)
\] 
and, hence,
\[
    (1+\Delta) \phi(x)^* (A^*A + \epsilon^2 I)^{-1} \phi(x) \geq k_n^\epsilon(x) \qquad \forall x \in X
\]
with high probability. Assuming $u \geq k_n^\epsilon$ holds with high probability, this ensures that $u^{(new)} \geq k_n^\epsilon$ with high probability as well. 

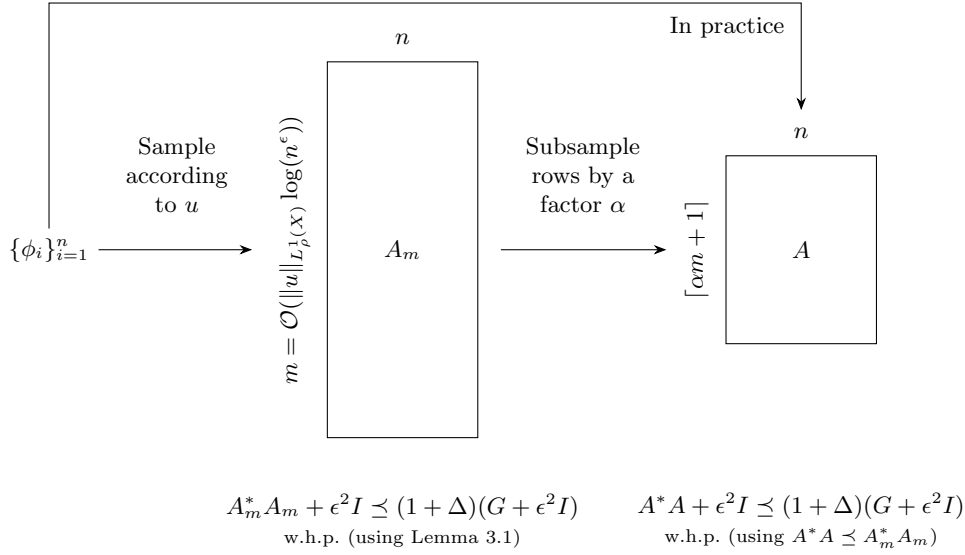
\begin{figure}
    \centering
    \begin{tikzpicture}[
        every node/.style={font=\small},
        matrix/.style={draw, minimum width=2cm, minimum height=5cm},
        reducedmatrix/.style={draw, minimum width=2cm, minimum height=2.5cm},
        >=Stealth
    ]
    \node (set) at (0,0) {$\{\phi_i\}_{i=1}^n$};
    \node at (1.7,1) {\parbox{2.2cm}{\centering Sample according \\ to $u$}};
    \draw[->] (set) -- ++(2.7,0);

    \node[matrix] (tildeA) at (4.7,0) {};
    \node[above=0.1cm of tildeA] {$n$};
    \node[rotate=90, anchor=south] at ([xshift=-0.1cm]tildeA.west) {$m = \mathcal{O}( \|u\|_{L^1_\rho(X)} \log(n^\epsilon))$};
    \node at (tildeA) {$A_m$};

    \node at (7.1,1) {\parbox{2.2cm}{\centering Subsample rows by a factor $\alpha$}};
    \draw[->] (tildeA.east) + (0.3,0) -- ++(2.5,0);
    \node[reducedmatrix] (hatA) at (10,0) {};
    \node[above=0.1cm of hatA] {$n$};
    \node[rotate=90, anchor=south] at ([xshift=-0.1cm]hatA.west) {$\lceil \alpha m + 1 \rceil$};
    \node at (hatA) {$A$};

    \node at (4.7,-3.6) {\parbox{5cm}{\centering $A_m^*A_m + \epsilon^2 I \preceq (1+\Delta)(G + \epsilon^2 I)$ \\ \begin{scriptsize} w.h.p.\ (using Lemma~\ref{lm:matrixineq}) \end{scriptsize}}};
    \node at (10,-3.6) {\parbox{4.5cm}{\centering $A^*A + \epsilon^2 I \preceq (1+\Delta)(G + \epsilon^2 I)$ \\ \begin{scriptsize} w.h.p.\ (using $A^*A \preceq A_m^*A_m$) \end{scriptsize}}};

    \draw[->] (set.north) -- ++(0,3) -| node[above, pos=0.9, xshift=-28pt, yshift=16pt, font=\small] {In practice} ([yshift=0.6cm]hatA.north);

    \end{tikzpicture}
    \caption{Schematic displaying the theoretical construction that motivates the definition of $A$ and, more specifically, why $u^{(new)} \geq k_n^\epsilon$ with high probability. Importantly, the matrix $A_m$ is never formed explicitly, so that only $\lceil \alpha m + 1 \rceil$ samples are drawn in each step.}
    \label{fig:schematic}
\end{figure}

How do we choose $\alpha$ so that $u$ improves significantly in each iteration while keeping the cost moderate? This is investigated in Section~\ref{sec:theory} and, more specifically, in Theorem~\ref{thm:iterativescheme}, which provides a bound on $\|u^{(new)}\|_{L^1_\rho(X)}$ as a function of $\alpha$. If we choose 
\[
    \alpha = \frac{C_2 n^\epsilon}{C_1 \|u\|_{L^1_\rho(X)}}
\]
for some constant $C_2 \geq 1$, then the amount of samples drawn in each iteration equals 
\[
    \lceil \alpha m + 1 \rceil = \lceil C_2 n^\epsilon \log(n^\epsilon) + 1 \rceil = \mathcal{O}(n^\epsilon \log(n^\epsilon))
\]
while
\[
    \|u^{(new)}\|_{L^1_\rho(X)} \leq \left( \frac{2}{1-\Delta} \right) \frac{n^\epsilon}{\alpha} = \left( \frac{2}{1-\Delta} \right) \frac{C_1}{C_2} \|u\|_{L^1_\rho(X)}
\]
with high probability. Thus, for given $C_1$ and $\Delta$, choosing $C_2$ appropriately ensures a geometric decrease in $L_\rho^1(X)$-norm of $u$, using only $\mathcal{O}(n^\epsilon \log(n^\epsilon))$ samples in each iteration. Note that we do not have to compute $m$ or $\alpha$; we simply draw $\lceil \alpha m + 1 \rceil$ samples in each iteration.

The complete procedure is given in Algorithm~\ref{alg}. A couple of remarks are in order. First, the bound above only applies when $\lceil \alpha m + 1 \rceil \leq m$, which roughly translates to $\alpha \leq 1$. Hence, we terminate the algorithm when $\alpha > 1$, i.e., when 
\[
    \|u\|_{L^1_\rho(X)} < \frac{C_2 n^\epsilon}{C_1}.
\]
At this point, the requirements in~\eqref{eq:urequirements} are satisfied for $C = C_2/C_1$ with high probability, implying that the number of samples outputted by the algorithm can be bounded from above by
\[
    C_1 \|u\|_{L^1_\rho(X)} \log(n^\epsilon) \leq C_2 n^\epsilon \log(n^\epsilon).
\]
Second, we exploit that $u$ can be evaluated efficiently for multiple values of $x$ via
\begin{equation} \label{eq:thinqr}
    \phi(x)^* (A^*A + \epsilon^2 I)^{-1} \phi(x) = \norm{ (R^*)^{-1} \phi(x) }_2^2,
\end{equation}
where
\[
    QR = \begin{bmatrix} A \\ \epsilon I_{n \times n} \end{bmatrix}
\]
is a (thin) QR-decomposition. Hence, $u$ is updated by storing $R$. 

{
\renewcommand{\baselinestretch}{1.3}\normalsize 
\begin{algorithm}
\caption{Refinement-based Christoffel Sampling (RCS)}
\label{alg}
  \begin{algorithmic}
    \Require $\Phi_n \subset L^2_\rho(X)$, (upper bounds on) $\|k_n^\epsilon\|_{L^\infty(X)}$ and $n^\epsilon$, parameters $C_1, C_2 > 0$ and $0 < \Delta < 1$
    \State $u(x) \gets \|k_n^\epsilon\|_{L^\infty(X)}$ 
    \While{$\|u\|_{L^1_\rho(X)} > \tfrac{C_2}{C_1} n^\epsilon$}
        \State Draw $\lceil C_2 n^\epsilon \log(n^\epsilon) + 1 \rceil$ samples $x_k$ i.i.d.\ from $\mu_u$~\eqref{eq:usampler}
        \State Construct matrix $A$: $(A)_{kj} \gets \overline{\phi_j(x_k)}/\sqrt{C_1 u(x_k) \log(n^\epsilon)}$ \vspace*{-5pt}
        \State \vspace*{-5pt} Update $u$~\eqref{eq:unew} via~\eqref{eq:thinqr} by computing the thin QR-decomposition of
        $ \begin{bmatrix} A \\ \epsilon I_{n \times n} \end{bmatrix}$
        \State Estimate $\|u\|_{L^1_\rho(X)}$
    \EndWhile
    \State Draw $C_1 \|u\|_{L^1_\rho(X)} \log(n^\epsilon)$ samples $x_k$ i.i.d.\ from $\mu_u$~\eqref{eq:usampler}
    \State Compute the weights $w_k = 1/(C_1 u(x_k) \log(n^\epsilon))$
    \State \Return sample points $\{x_1, x_2, \dots \}$ and weights $\{w_1, w_2, \dots\}$
  \end{algorithmic}
\end{algorithm}
}

\subsection{Discussion} \label{sec:discussion}
The computational cost of the algorithm is
\begin{align*}
    &\underbrace{\mathcal{O}(\log(\|k_n^\epsilon\|_{L^\infty(X)}))}_{\text{\# iterations}} \; \times \underbrace{\mathcal{O}(n^\epsilon \log(n^\epsilon))}_{\text{\# samples per iteration}} \times \; \underbrace{\mathcal{O}(n^2 \log(\|k_n^\epsilon\|_{L^\infty(X)}))}_{\text{cost of evaluating $u$}} \; \; \text{flops} \\ = \; &\mathcal{O}(n^3 \log(n) \log(\|k_n^\epsilon\|_{L^\infty(X)})^2) \; \; \text{flops}
\end{align*}
where we used $n^\epsilon \leq n$. Here, the number of iterations depends logarithmically on $\|k_n^\epsilon\|_{L^\infty(X)}$, since $\|u\|_{L^1_\rho(X)}$ decreases geometrically and equals $\|k_n^\epsilon\|_{L^\infty(X)}$ at initialization. We also assume that drawing one sample from $\mu_u$~\eqref{eq:usampler} requires a constant number of evaluations of the function $u$, as discussed further below.

One evaluation of $u$ in iteration $i$ corresponds to computing 
\[
    u(x) = \min\left( \|k_n^\epsilon\|_{L^\infty(X)}, \; \norm{\left(R_{(1)}^* \right)^{-1} \phi(x)}_2^2, \; \dots, \; \norm{\left(R_{(i-1)}^* \right)^{-1} \phi(x)}_2^2 \right),
\]
where each system solve requires $\mathcal{O}(n^2)$ flops since $R$ is upper-triangular. We find that the computational cost of evaluating $u$ can be reduced in practice, as discussed in Section~\ref{sec:numericalexperiments}. The total cost of computing the $R$-factors equals
\[
    \underbrace{\mathcal{O}(\log(\|k_n^\epsilon\|_{L^\infty(X)}))}_{\text{\# iterations}} \; \times \underbrace{\mathcal{O}(n^\epsilon \log(n^\epsilon) n^2)}_{\text{\parbox{2.5cm}{\centering cost of QR \\ \vspace{-1mm} decomposition}}} = \mathcal{O}(n^3 \log(n) \log(\|k_n^\epsilon\|_{L^\infty(X)})) \; \text{flops},
\]
as this only needs to be done once per iteration.

We now discuss the assumptions underlying the algorithm. A further discussion on the practical implementation of the algorithm can be found in Section~\ref{sec:numericalexperiments}.
\begin{description}
    \item[Known upper bound on $\boldsymbol{\|k_n^\epsilon\|_{L^\infty(X)}}$] 
    For polynomial spaces, the growth of the Christoffel function on irregular domains has been studied extensively~\cite{kroo2015christoffel,dolbeault2022optimal}. Even if one does not have a tight upper bound, overestimating it has little impact since the computing time depends only logarithmically on this bound.
    
    \item[Known upper bound on $\boldsymbol{n^\epsilon}$] A natural choice for an upper bound on the numerical dimension $n^\epsilon$ is $n$, since $n^\epsilon \leq n$ for all $\epsilon \geq 0$. In this case, the algorithm produces a sample set of size $\mathcal{O}(n \log(n))$, which suffices in most practical scenarios. If the set $\Phi_n$ is highly numerically redundant, however, it can be advantageous to use a numerical dimension $n^\epsilon \ll n$. 
    
    If such an estimate is unknown, one might attempt to first run Algorithm~\ref{alg} with $n$ as an upper bound, then estimate the numerical dimension using the resulting discrete Gram matrix $G_m$ and, if the estimate is significantly smaller than $n$, re-run the algorithm using this refined estimate. Unfortunately, this approach is unreliable. To see why, note that the discrete Gram $G_m$ is guaranteed to satisfy~\eqref{eq:matrixineq} with high probability. In the specific case where $G_m + \epsilon^2 I = (1-\Delta) (G+\epsilon^2 I)$, straightforward calculations show that
    \[
        \sum_{i=1}^n \frac{\lambda_i(G_m)}{\lambda_i(G_m) + \epsilon^2} = \Tr(G_m (G_m + \epsilon^2 I)^{-1}) = n^\epsilon - \frac{\Delta}{1-\Delta} \sum_{j=1}^n \frac{\epsilon^2}{\epsilon^2 + \lambda_i(G)}.
    \]
    For highly redundant bases, this results in a significant underestimation of $n^\epsilon$, causing the algorithm to fail. Reliable estimates can instead be obtained when a condition analogous to~\eqref{eq:matrixineq} holds for some $\widetilde{\epsilon} \ll \epsilon$, i.e., by using multiple regularization parameters. We do not pursue this approach further.

    \item[Sampling from $\boldsymbol{\mu_u}$ and estimating $\boldsymbol{\|u\|_{L^1_\rho(X)}}$] 
    The algorithm assumes that samples of $\mu_u$ and estimates of $\|u\|_{L^1_\rho(X)}$ can be obtained efficiently from evaluations of $u$. In practice, we find that Markov chain samplers yield satisfactory results, even though the samples are not independent. Moreover, these samplers can often run without knowledge of the normalization factor of $\mu_u$, i.e., $\|u\|_{L^1_\rho(X)}$, meaning that only a rough estimate of $\|u\|_{L^1_\rho(X)}$ is needed. This estimate is then only used to terminate Algorithm~\ref{alg}. We find that a simple Monte Carlo estimator or Matlab's built-in \texttt{integral} functions work well.
\end{description}

\section{Theoretical results} \label{sec:theory}
In this section, we establish that the proposed algorithm converges with high probability. We begin by stating the main result---a corollary of Theorem~\ref{thm:iterativescheme}, proved below, together with the algorithmic description given above.

\begin{corollary}
    Let $X \subseteq \mathbb{R}^d$ be compact, $\Phi = \{\phi_i\}_{i=1}^n \subset L^2_\rho(X)$ be a set of piecewise Lipschitz continuous functions, $\epsilon > 0$ and assume $\inf_{x \in X} k^\epsilon(x) > 0$. Run Algorithm~\ref{alg} for some $0 < \Delta < 1$, upper bounds $K \geq \|k_n^\epsilon\|_{L^\infty}$ and $N \geq n^\epsilon$, sufficiently large $C_1$ and for some $C_2$ satisfying 
    \[
        C_2 > \frac{2C_1}{1 - \Delta} \qquad \text{s.t.} \qquad C = \frac{(1-\Delta)C_2}{2C_1} > 1.
    \]
    Then, with high probability (depending on $C_1$), the algorithm terminates in at most $\max\left(0, \left\lceil \log_{C}\left(\frac{(1-\Delta)K}{2N}\right) - 1 \right\rceil\right)$ iterations and produces at most $C_2 N \log(N)$ samples and weights whose associated discrete Gram matrix satisfies~\eqref{eq:matrixineq}, thereby ensuring accurate least squares fitting as described in corollaries~\ref{corr1} and~\ref{corr2}.
\end{corollary}

The key step in proving this result is to show that, in each step of the algorithm, the updated function $u^{(new)}$ defined in~\eqref{eq:unew} remains an upper bound for $k_n^\epsilon$, while its $L^1_\rho(X)$-norm decreases relative to that of $u$. This property is quantified precisely in Theorem~\ref{thm:iterativescheme}. The algorithm presented in this paper can be viewed as a continuous analogue of the iterative leverage score sampling algorithm from~\cite[Algorithm 3]{cohen2015uniform}, while incorporating $\ell^2$-regularization as in~\cite{cohen2017input}. The analysis relies heavily on this connection.

For clarity, we use a more explicit notation throughout this section. More specifically, we use
\begin{align*}
    k^\epsilon(x; \Psi) &= \psi(x)^* (G + \epsilon^2 I)^{-1} \psi(x) \\
    k^\epsilon(x; \Psi, M) &= \psi(x)^*(M + \epsilon^2 I)^{-1} \psi(x) \\
    n^\epsilon(\Psi) &= \int_X k^\epsilon(x; \Psi) d\rho = \sum_{i=1}^n \frac{\lambda_i(G)}{\lambda_i(G) + \epsilon^2}
\end{align*}
where $\psi(x) = \begin{bmatrix} \psi_1(x) & \dots & \psi_n(x) \end{bmatrix}^\top$ and $G$ is the continuous Gram matrix associated with $\Psi = \{\psi_i\}_{i=1}^n$. 

\begin{theorem} \label{thm:iterativescheme}
Let $X \subseteq \mathbb{R}^d$ be compact, $\Phi = \{\phi_i\}_{i=1}^n \subset L^2_\rho(X)$ be a set of piecewise Lipschitz continuous functions, $\epsilon > 0$ and assume $\inf_{x \in X} k^\epsilon(x; \Phi) > 0$. Let $u$ be an integrable, piecewise Lipschitz function satisfying
\[
    u(x) \ge k^\epsilon(x; \Phi), \quad \forall x \in X.
\]
For some \(\alpha > 0\), \(\Delta > 0\), \(\gamma < 1\), let
\[
    m \ge c_\Delta \log\big(16 n^\epsilon(\Phi)/\gamma \big) \int_X u \, d\rho, 
    \qquad c_\Delta = \frac{1 + 2\Delta/3}{\Delta^2/2},
\]
for small enough $\alpha$ such that $\lceil \alpha m + 1 \rceil < m$. Draw i.i.d.\ samples $\{x_k\}_{k=1}^{\lceil \alpha m + 1\rceil}$ from $\mu_u$~\eqref{eq:usampler} and define the $\lceil \alpha m + 1\rceil \times n$ matrix $A$ via
\[
(A)_{kj} = \sqrt{\frac{w_u(x_k)}{m}} \overline{\phi_j(x_k)}.
\]
Then, with probability at least $1-2\gamma$, the updated function $u^{(new)}$~\eqref{eq:unew} satisfies
\begin{align}
    u^{(new)}(x) &\ge k^\epsilon(x; \Phi), \quad \forall x \in X, \label{thm:iterativescheme1} \\
    \int_X u^{(new)} \, d\rho &\le \left(\frac{2}{1-\Delta}\right) \frac{n^\epsilon(\Phi)}{\alpha}. \label{thm:iterativescheme2}
\end{align}
\end{theorem}

\begin{proof}
    \noindent \underline{1. Establishing~\eqref{thm:iterativescheme1}} \smallskip \\
   The matrix $A$ can be viewed as a submatrix of an $m \times n$ matrix $A_m$, equivalently defined by
    \[
        (A_m)_{ij} = \sqrt{\frac{w_u(x_k)}{m}} \overline{\phi_j(x_k)}.
    \]
    Due to Lemma~\ref{lm:matrixineq}, we have $\frac{1}{1+\Delta}(A_m^*A_m + \epsilon^2 I) \preceq G + \epsilon^2 I$ with probability at least $1-\gamma$. Furthermore, 
    \begin{equation*}
        A_m^*A_m - A^*A = \sum_{k=\lceil \alpha m + 1\rceil + 1}^{m} \frac{w(x_k)}{m} \phi(x_k) \phi(x_k)^* \succeq 0 \qquad \Rightarrow \qquad A^*A \preceq A_m^*A_m
    \end{equation*}
    and, hence, $\frac{1}{1+\Delta}(A^*A + \epsilon^2 I) \preceq G + \epsilon^2 I$ with probability $1-\gamma$ as well. It follows that 
    \[ 
        (1+\Delta) k^{\epsilon} \left(x; \Phi, A^*A \right) \geq k^\epsilon\left(x; \Phi, G \right) = k^\epsilon(x; \Phi)
    \]
    and, using the definition of $u^{(new)}$ and the fact that $u \geq k^\epsilon(x; \Phi)$,
    \[
        u^{(new)}(x) \geq k^\epsilon(x; \Phi) \qquad \forall x \in X
    \]
    as required. \medskip

    \noindent \underline{2. Establishing~\eqref{thm:iterativescheme2}} \smallskip \\
    It follows from Lemma~\ref{lm:weighting}, presented below, that for every $\delta_1, \delta_2 > 0$ there exists a function $v$ such that 
    \[
        \alpha u(x) + \delta_1 \geq k^\epsilon(x; v\Phi), \quad \forall x \in X \qquad \text{and} \qquad \int_X \alpha u \; 1_{\{v(x) \neq 1\}} \; d\rho \leq n^\epsilon(\Phi) + \delta_2. 
    \]
    Using this $v$, we can derive that
    \begin{align}
        \int_X u^{(new)} \; d\rho 
        &= \int_X u^{(new)} \; 1_{\{v(x) \neq 1\}} \; d\rho + \int_X  u^{(new)} \; 1_{\{v(x) = 1\}} \; d\rho \nonumber \\
        &\leq \int_X u \; 1_{\{v(x) \neq 1\}} \; d\rho + \int_X (1+\Delta) k^\epsilon\left(x; \Phi, A^*A\right) \; 1_{\{v(x) = 1\}} \; d\rho \nonumber \\
        &\leq \frac{n^\epsilon(\Phi) + \delta_2}{\alpha} + \int_X (1+\Delta) k^\epsilon\left(x; \Phi, A^*A\right) \; 1_{\{v(x) = 1\}} \; d\rho \nonumber \\
        &= \frac{n^\epsilon(\Phi) + \delta_2}{\alpha} + \int_X (1+\Delta) k^\epsilon\left(x; v\Phi, A^*A\right) \; 1_{\{v(x) = 1\}} \; d\rho, \label{thm:iterativescheme:boundint}
    \end{align}
    where $v\Phi = \{v\phi_i\}_{i=1}^n$. 
    
    We proceed by bounding $k^\epsilon\left(x; v\Phi, A^*A\right)$. To this end, note that $0 \leq v(x) \leq 1$ implies $V^*V \preceq A^*A$, where 
    \[
        (V)_{kj} = \sqrt{\frac{w_u(x_k)}{m}} v(x_k) \overline{\phi_j(x_k)} = v(x_k) (A)_{kj}.
    \]
    Furthermore, since $\alpha u(x) \geq k^\epsilon(x; v\Phi) + \delta_1$, we have 
    \begin{equation} \label{thm:iterativescheme:eq1}
         (1-\Delta) (G^v + \epsilon^2 I) \preceq \frac{m}{\lceil \alpha m + 1\rceil} V^*V + \epsilon^2 I \qquad \text{ where } (G^v)_{ij} = \langle v\phi_i, v\phi_j \rangle_{L^2_\rho(X)}
    \end{equation}
    with probability at least $1-\gamma$ using a similar reasoning as in Lemma~\ref{lm:matrixineq} for the set $v\Phi$. More specifically, note that the samples used to construct $V$ are drawn i.i.d.\ from $\mu_u$. In order for~\eqref{thm:iterativescheme:eq1} to hold, we need that the number of samples $\lceil \alpha m + 1 \rceil$ satisfies
    \begin{align}
        \begin{split} \label{thm:iterativescheme:eq2}
        \lceil \alpha m + 1 \rceil &\geq c_\Delta \log(16n^\epsilon(v\Phi)/\gamma) \esssup_{x \sim \rho} \left( \frac{k^\epsilon(x;v\Phi) }{\alpha u(x)} \int_X \alpha u \; d\rho\right) \\ 
        &= \alpha c_\Delta \log(16n^\epsilon(v\Phi)/\gamma) \int_X u d\rho \esssup_{x \sim \rho} \left( \frac{k^\epsilon(x;v\Phi) }{\alpha u(x)} \right).
        \end{split}
    \end{align}
    Due to Lemma~SM2.1, it suffices that 
    \[
        \lceil \alpha m + 1 \rceil \geq \alpha c_\Delta \log(16n^\epsilon(\Phi)/\gamma) \int_X u d\rho \esssup_{x \sim \rho} \left( \frac{k^\epsilon(x;v\Phi) }{\alpha u(x)} \right).
    \]
    By construction we have
    \[ 
        m \geq c_\Delta \log(16n^\epsilon(\Phi)/\gamma) \int_X u d\rho,
    \] 
    and, furthermore,
    \begin{align*}
        \esssup_{x \sim \rho} \left( \frac{k^\epsilon(x;v\Phi) }{\alpha u(x)} \right) &\leq \esssup_{x \sim \rho} \left( \frac{k^\epsilon(x;v\Phi)}{\alpha u(x)+\delta_1 } \right) \esssup_{x \sim \rho} \left( \frac{\alpha u(x) + \delta_1 }{\alpha u(x) }\right) \\
        &\leq 1 + \frac{\delta_1 }{\essinf_{x \sim \rho} \alpha u(x) }.
    \end{align*}
    If we consider some $\delta_1 > 0$ satisfying
    \[ 
        \delta_1 \leq \frac{\essinf_{x \sim \rho} \alpha u(x)}{\alpha c_\Delta \log(16n^\epsilon(\Phi)/\gamma) \int_X u d\rho} = \frac{\essinf_{x \sim \rho} u(x)}{c_\Delta \log(16n^\epsilon(\Phi)/\gamma) \int_X u d\rho},
    \]
    it follows that~\eqref{thm:iterativescheme:eq2} and, hence,~\eqref{thm:iterativescheme:eq1} is satisfied. Note that this is possible since $\essinf_{x \sim \rho} u(x) > 0$ due to $\inf_{x \in X} k^\epsilon(x; \Phi) > 0$. In conclusion, we obtain the bound
    \begin{align*}
        k^\epsilon \left(x; v\Phi, A^*A \right) 
        &\leq k^\epsilon \left(x; v\Phi, V^*V \right) \qquad &(V^*V \preceq A^*A) \\
        &\leq \frac{m}{\lceil \alpha m + 1\rceil} k^\epsilon \left(x; v\Phi, \frac{m}{\lceil \alpha m + 1\rceil} V^*V \right)  \qquad &\text{(Lemma~SM2.2)}\\
        &\leq \frac{m}{\lceil \alpha m + 1\rceil} \left( \frac{1}{1-\Delta} \right)  k^\epsilon(x; v\Phi, G^v) \qquad &\text{\eqref{thm:iterativescheme:eq1}} \\
        &\leq \frac{1}{\alpha} \left( \frac{1}{1-\Delta} \right) k^\epsilon(x;v\Phi).
    \end{align*}
    
    Using this bound in~\eqref{thm:iterativescheme:boundint}, we arrive at
    \begin{align*}
        \int_X u^{(new)} \; d\rho 
        &\leq \frac{n^\epsilon(\Phi) + \delta_2}{\alpha} + \int_X \frac{1}{\alpha} \left( \frac{1+\Delta}{1-\Delta} \right) k^\epsilon(x; v \Phi) \; 1_{\{v(x) = 1\}} \; d\rho \\
        &\leq \frac{n^\epsilon(\Phi) + \delta_2}{\alpha} +\frac{1}{\alpha} \left( \frac{1+\Delta}{1-\Delta} \right) n^\epsilon(v\Phi) \\
        &\leq \left( 1 + \left( \frac{1+\Delta}{1-\Delta} \right) \right) \frac{n^\epsilon(\Phi)}{\alpha} + \frac{\delta_2}{\alpha}\\
        &= \left(\frac{2}{1-\Delta}\right) \frac{n^\epsilon(\Phi)}{\alpha} + \frac{\delta_2}{\alpha},
    \end{align*}
    where in the third step we use Lemma~SM1.1. Since this bound holds for every $\delta_2 > 0$, one has 
    \begin{align*}
        \int_X u^{(new)} \; d\rho 
        &\leq \inf_{\delta_2 > 0} \left(\frac{2}{1-\Delta}\right) \frac{n^\epsilon(\Phi)}{\alpha} + \frac{\delta_2}{\alpha} = \left(\frac{2}{1-\Delta}\right) \frac{n^\epsilon(\Phi)}{\alpha}.
    \end{align*}

    \noindent \underline{3. Establishing the probability} \smallskip \\
    The probability that the events $E_1 = \{\text{\eqref{thm:iterativescheme1} holds}\}$ and $E_2 = \{\text{\eqref{thm:iterativescheme2} holds}\}$ occur simultaneously is bounded by
    \[
    \mathbb{P}(E_1 \cap E_2) = 1 - \mathbb{P}(E_1^c \cup E_2^c) \geq 1 - \mathbb{P}(E_1^c) - \mathbb{P}(E_2^c) \geq 1 - 2\gamma.
    \]
\end{proof}

The proof of Theorem~\ref{thm:iterativescheme} relies heavily on the following lemma, which asserts the existence of a weight function $v$ (with $0 \leq v(x) \leq 1, \forall x \in X$) such that the numerical inverse Christoffel function $k_n^\epsilon$ associated with the weighted set $v\Phi$ lies below a prescribed function $u$ (up to an arbitrarily small constant $\delta_1 > 0$). Moreover, the function $v$ does not distort the functions in $\Phi$ too much, in the sense of~\eqref{lm:weighting:eq2}. Note that without condition~\eqref{lm:weighting:eq2}, condition~\eqref{lm:weighting:eq1} would be trivial to satisfy---for instance, by choosing $v = 0$.

This lemma is closely related to discrete analogues, namely~\cite[Theorem 2]{cohen2015uniform} (for leverage scores) and~\cite[Lemma 6.2]{cohen2017input} (for ridge leverage scores with variable regularization parameter). On top of this connection, our proof proceeds by first discretizing the problem and then applying the proof techniques from~\cite[Theorem 2]{cohen2015uniform} and~\cite[Lemma 6.2]{cohen2017input}. Specifically, we require a discrete result on ridge leverage scores with a fixed regularization parameter, which is provided in the supplementary material.

\begin{lemma} \label{lm:weighting}
    Let $X \subset \mathbb{R}^n$ be compact, $\{\phi_i\}_{i=1}^n \subset L^2_\rho(X)$ be a set of piecewise Lipschitz continuous functions, $u: X \to \mathbb{R}$ be a strictly positive and piecewise Lipschitz continuous function and $\epsilon > 0$. Then, for every $\delta_1,\delta_2 > 0$ there exists a function $v: X \to \mathbb{R}$ satisfying 
    \[
        0 \leq v(x) \leq 1 \qquad \forall x \in X
    \]
    such that 
    \begin{equation} \label{lm:weighting:eq1}
        k^\epsilon(x; v \Phi) \leq u(x) + \delta_1 \qquad \forall x \in X
    \end{equation}
    where $v\Phi = \{v \phi_i\}_{i=1}^n$, and 
    \begin{equation} \label{lm:weighting:eq2}
        \int_{X} u \; 1_{\{v(x) \neq 1\}} \; d\rho \leq n^\epsilon(\Phi) + \delta_2.
    \end{equation}
\end{lemma}
\begin{proof}
     We start by partitioning $X$ in $\ell$ $\rho$-measurable partitions $\{X_k\}_{k=1}^\ell$, identifying $\ell$ discretization points $\{x_k\}_{k=1}^\ell$ that satisfy $x_k \in X_k$, and constructing the matrix $B^{(\ell)} \in \mathbb{R}^{\ell \times n}$ via
    \[
        (B^{(\ell)})_{kj} = \sqrt{\rho(X_k)} \kern3pt \overline{\phi_j(x_k)}.
    \]
    We assume that the partitions are chosen in such a way that both the functions $\{\phi_i\}_{i=1}^n$ and the function $u$ are continuous in all partitions $X_k$. Furthermore, we define
    \[
        h^{(\ell)} = \max_{1 \leq k \leq \ell}\max_{x \in X_k} \left\vert x - x_k \right\vert,
    \]
    and
    \[
        u_k^{(\ell)} = \int_{X_k} u \; d\rho, \qquad 1 \leq k \leq \ell.
    \]
    Since $u$ is strictly positive, $u_k^{(\ell)} > 0$ for all $i$. From Lemma~SM1.1, it follows that there exists a diagonal matrix $V^{(\ell)} \in \mathbb{R}^{\ell \times \ell}$ with $0 \preceq V^{(\ell)} \preceq I$ such that 
    \begin{equation} \label{lm:weighting:proof1}
        \tau_k^\epsilon(V^{(\ell)}B^{(\ell)}) \leq u_k^{(\ell)}, \qquad 1 \leq k \leq \ell
    \end{equation}
    and
    \begin{equation} \label{lm:weighting:proof2}
        \sum_{V^{(\ell)}_{kk} \neq 1} u_k^{(\ell)} \leq n^\epsilon(B^{(\ell)}).
    \end{equation}
    Using this matrix $V^{(\ell)}$ we can define a piecewise continuous function $v^{(\ell)}$ by 
    \[
        v^{(\ell)}(x) = \sum_{k=1}^\ell V^{(\ell)}_{kk} 1_{X_k},
    \]
    where $1_{X_k}$ is the indicator function of $X_k$. \medskip

    \noindent \underline{1. Establishing~\eqref{lm:weighting:eq2}} \smallskip \\
    From~\eqref{lm:weighting:proof2}, it follows immediately that
    \[  
        \int_{X} u \; 1_{\{v^{(\ell)}(x) \neq 1\}} \; d\rho = \sum_{k: V^{(\ell)}_{kk} \neq 1} \int_{X_k} u \; d\rho = \sum_{k: V^{(\ell)}_{kk} \neq 1} u_k^{(\ell)} \leq n^\epsilon(B^{(\ell)})
    \]
    In order to relate $n^\epsilon(B^{(\ell)})$ to $n^\epsilon(\Phi)$, consider 
    \begin{align}
        \begin{split} \label{lm:weighting:grams}
            \left\vert ((B^{(\ell)})^* B^{(\ell)})_{ij} - (G)_{ij} \right\vert &= \left\vert \sum_{k=1}^m \rho(X_k) \phi_i(x_k) \overline{\phi_j(x_k)} - \int_{X} \phi_i \overline{\phi_j} \; d\rho \right\vert \\
            &\leq \sum_{k=1}^m \left\vert \rho(X_k)  \phi_i(x_k) \overline{\phi_j(x_k)} - \int_{X_k} \phi_i \overline{\phi_j} \; d\rho \right\vert \\
            &= \sum_{k=1}^m \left\vert \int_{X_k} \phi_i(x_k) \overline{\phi_j(x_k)} - \phi_i \overline{\phi_j} \; d\rho \right\vert \\
            &\leq C \rho(X) h^{(\ell)},
        \end{split}
    \end{align}
    for some $C > 0$, where in the last step we use that the functions $\phi_i$ are Lipschitz continuous in every partition and bounded. Using~\cite[\S 21-1]{hogben2006handbook}, it follows that there exists a permutation $\tau$ over $\{1,2,\dots,n\}$ such that 
    \[
        \lvert \lambda_i(G) - \lambda_{\tau(i)}((B^{(\ell)})^* B^{(\ell)}) \rvert \leq \|G - (B^{(\ell)})^* B^{(\ell)}\|_F \leq C n \rho(X) h^{(\ell)}, \qquad 1 \leq i \leq n
    \]
    and, hence, 
    \[
        n^\epsilon(B^{(\ell)}) \leq n^\epsilon(\Phi) + C n^2 \rho(X) h^{(\ell)} \epsilon^{-2}
    \]
    since $\lambda / (\lambda + \epsilon^2)$ is monotonically increasing for nonnegative $\lambda$ and 
    \[  
        \frac{\lambda + \delta}{\lambda + \delta + \epsilon^2} \leq \frac{\lambda}{\lambda + \epsilon^2} + \epsilon^{-2} \delta
    \]
    for $\delta \geq 0$. Because it is possible to construct partitions such that $h^{(\ell)} \to 0$ as $\ell \to \infty$, we obtain that for every $\delta_2 > 0$ there exists a sufficiently large $\ell$ such that 
    \[
        \int_{X} u \; 1_{\{v(x) \neq 1\}} \; d\rho \leq n^\epsilon(\Phi) + \delta_2
    \]
    for $v = v^{(\ell)}$, as required. \medskip
    
    \noindent \underline{2. Establishing~\eqref{lm:weighting:eq1}} \smallskip \\
    We divide~\eqref{lm:weighting:proof1} by $\rho(X_k)$ and consider
    \begin{align} 
        \begin{split} \label{lm:weighting:proof3} 
            \tau^\epsilon_i(V^{(\ell)}B^{(\ell)})/\rho(X_k) &\leq u_k^{(\ell)}/\rho(X_k) \\ 
            &\Updownarrow \\
            k^\epsilon(x; v^{(\ell)}\Phi) + \underbrace{\tau^\epsilon_k(V^{(\ell)}B^{(\ell)})/\rho(X_k) - k^\epsilon(x; v^{(\ell)}\Phi)}_{\text{error 1}} &\leq u(x) + \underbrace{u_k^{(\ell)}/\rho(X_k) - u(x)}_{\text{error 2}}
        \end{split}
    \end{align}
    for $x \in X_k$. We will show that both errors in the inequality above go to zero as $h^{(\ell)}$ goes to zero, for every $k$. Because it is possible to construct partitions such that $h^{(\ell)} \to 0$ as $\ell \to \infty$, we obtain that for every $\delta_1 > 0$ there exists a sufficiently large $\ell$ such that
    \[
        k^\epsilon(x; v\Phi) \leq u(x) + \delta_1, \qquad \forall x \in X
    \]
    for $v = v^{(\ell)}$, as required
    
    The first error can be bounded by splitting it as
    \begin{align} \label{lm:weighting:proof4}
    \begin{split} \left\vert \tau^\epsilon_k(V^{(\ell)}B^{(\ell)})/\rho(X_k) - k^\epsilon(x; v^{(\ell)}\Phi) \right\vert \;  \leq \; 
    &\left\vert \tau^\epsilon_k(V^{(\ell)}B^{(\ell)})/\rho(X_k) - k\left(x;v^{(\ell)}\Phi,G^{v^{(\ell)}}_\ell\right) \right\vert \\ &+ \left\vert k\left(x;v^{(\ell)}\Phi,G^{v^{(\ell)}}_\ell\right) - k^\epsilon(x; v^{(\ell)}\Phi) \right\vert 
    \end{split}
    \end{align}
    where $G^{v^{(\ell)}}_\ell = (V^{(\ell)}B^{(\ell)})^* V^{(\ell)}B^{(\ell)}$. Since $x \in X_k \Rightarrow v^{(\ell)}(x) = V^{(\ell)}_{kk}$, we can bound the first term by
    \begin{align*}
        &\left\vert \tau_k^\epsilon(V^{(\ell)}B^{(\ell)})/\rho(X_k) - k\left(x;v^{(\ell)}\Phi,G^{v^{(\ell)}}_\ell\right) \right\vert \\ &= \left\vert v^{(\ell)}(x_k)^2\phi(x_k)^*(G^{v^{(\ell)}}_\ell + \epsilon^2 I)^{-1}\phi(x_k) - v^{(\ell)}(x)^2\phi(x)^*(G^{v^{(\ell)}}_\ell + \epsilon^2 I)^{-1}\phi(x) \right\vert \\
        &= \left(V^{(\ell)}_{kk}\right)^2 \left\vert (\phi(x_k)-\phi(x))^*(G^{v^{(\ell)}}_\ell + \epsilon^2 I)^{-1}(\phi(x_k) + \phi(x)) \right\vert \\
        &\leq \|\phi(x_k) - \phi(x)\|_2 \|(G^{v^{(\ell)}}_\ell + \epsilon^2 I)^{-1}\|_2 \|\phi(x_k) + \phi(x)\|_2 \\
        &\leq C n^2 h^{(\ell)} \epsilon^{-2},
    \end{align*}
    for some $C > 0$, where in the last step we use that the functions $\phi_i$ are Lipschitz continuous in $X_k$ and bounded. This term indeed goes to zero as $h^{(\ell)} \to 0$. 
    
    In order to bound the second part of~\eqref{lm:weighting:proof4}, define $G^{v^{(\ell)}} \in \mathbb{C}^{n \times n}$ as 
    \[
        \left(G^{v^{(\ell)}}\right)_{ij} = \langle v^{(\ell)} \phi_i, v^{(\ell)} \phi_j \rangle_{L^2_\rho(X)}
    \]
    and consider 
    \begin{align*}
        \left\vert \left(G^{v^{(\ell)}}_\ell\right)_{ij} - \left(G^{v^{(\ell)}}\right)_{ij} \right\vert &= \left\vert \sum_{k=1}^\ell \rho(X_k) v^{(\ell)}(x_k)^2 \phi_i(x_k) \overline{\phi_j(x_k)} - \int_{X} (v^{(\ell)})^2 \phi_i \overline{\phi_j} \; d\rho \right\vert \\
        &\leq \sum_{k=1}^\ell \left\vert \rho(X_k) v^{(\ell)}(x_k)^2 \phi_i(x_k) \overline{\phi_j(x_k)} - \int_{X_k} (v^{(\ell)})^2 \phi_i \overline{\phi_j} \; d\rho \right\vert \\
        &= \sum_{k=1}^\ell \left(V^{(\ell)}_{kk}\right)^2 \left\vert \rho(X_k)  \phi_i(x_k) \overline{\phi_j(x_k)} - \int_{X_k} \phi_i \overline{\phi_j} \; d\rho \right\vert \\
        &\leq \sum_{k=1}^\ell \left\vert \rho(X_k) \phi_i(x_k) \overline{\phi_j(x_k)} - \int_{X_k} \phi_i \overline{\phi_j} \; d\rho \right\vert \\
        &\leq C \rho(X) h^{(\ell)},
    \end{align*}
    for some $C>0$, where the last step is equivalent to the steps in~\eqref{lm:weighting:grams}. This derivation indeed allows us to bound the second term of~\eqref{lm:weighting:proof4} for all $x \in X_k$, namely
    \begin{align*}
        &\left\vert k\left(x;v^{(\ell)}\Phi,G^{v^{(\ell)}}_\ell\right) - k^\epsilon(x; v^{(\ell)}\Phi) \right\vert \\ &= \left\vert v^{(\ell)}(x)^2\phi(x)^*(G^{v^{(\ell)}}_\ell+\epsilon^2 I)^{-1}\phi(x) - v^{(\ell)}(x)^2\phi(x)^*(G^{v^{(\ell)}}+\epsilon^2 I)^{-1}\phi(x) \right\vert \\
        &= \left(V_{kk}^{(\ell)}\right)^2 \left\vert \phi(x)^*(G^{v^{(\ell)}}_\ell+\epsilon^2 I)^{-1}\phi(x) - \phi(x)^*(G^{v^{(\ell)}}+\epsilon^2 I)^{-1}\phi(x) \right\vert \\
        &\leq \left\vert \phi(x)^*(G^{v^{(\ell)}}_\ell+\epsilon^2 I)^{-1}\phi(x) - \phi(x)^*(G^{v^{(\ell)}}+\epsilon^2 I)^{-1}\phi(x) \right\vert \\
        &= \left\vert \phi(x)^*((G^{v^{(\ell)}}_\ell+\epsilon^2 I)^{-1} ((G^{v^{(\ell)}}+\epsilon^2 I) - (G^{v^{(\ell)}}_\ell+\epsilon^2 I))(G^{v^{(\ell)}}+\epsilon^2 I)^{-1})\phi(x) \right\vert \\
        &\leq \|\phi(x)\|^2_2 \|(G^{v^{(\ell)}}_\ell+\epsilon^2 I)^{-1} \|_2  \|G^{v^{(\ell)}} - G^{v^{(\ell)}}_\ell\|_2 \|(G^{v^{(\ell)}}+\epsilon^2 I)^{-1}\|_2 \\
        &\leq C \epsilon^{-4} n^{5/2} \rho(X) h^{(\ell)},
    \end{align*}
    for some $C > 0$, where in the last step we use that the functions $\phi_i$ are bounded and $\|M\|_2 \leq \sqrt{n} \max_{1 \leq i \leq n} \sum_{j=1}^n \lvert (M)_{ij} \rvert$ for $M \in \mathbb{C}^{n \times n}$~\cite[\S 50-4]{hogben2006handbook}. This term also goes to zero as $h^{(\ell)} \to 0$ and, hence, the first error in~\eqref{lm:weighting:proof3} goes to zero as $h^{(\ell)} \to 0$. 
    
    Similarly for the second error in~\eqref{lm:weighting:proof3}, we get for all $x \in X_k$ that
    \begin{align*}
        \left\vert u_i^{(\ell)}/\rho(X_k) - u(x) \right\vert 
        &= \left\vert \int_{X_k} u(y) - u(x) \; d\rho(y) \right\vert / \rho(X_k) \\
        &\leq \sup_{y \in X_k} \lvert u(y) - u(x) \rvert \\
        &\leq Ch^{(\ell)},
    \end{align*}
    for some $C > 0$, where in the last step we use that $u$ is Lipschitz continuous in $X_k$.
\end{proof}

\section{Numerical experiments} \label{sec:numericalexperiments}
In this section we verify the performance of the refinement-based Christoffel sampling (RCS) algorithm numerically. Its implementation in \texttt{MATLAB} and \texttt{Julia} and all code needed to reproduce the numerical experiments are made publicly available~\cite{github_repo}. We use slice sampling~\cite{neal2003slice} to (approximately) draw from $\mu_u$ and a simple Monte Carlo estimate to approximate $\|u\|_{L_\rho^1(X)}$. 

Some changes were made in the practical implementation compared to Algorithm~\ref{alg}, which favor the computational complexity while the empirical performance remains satisfactory. First, we discard $\log$-terms. More specifically, in each iterations we sample $C_2 n^\epsilon$ points and construct the matrix $A$ via $(A)_{kj} = \overline{\phi_j(x_k)}/\sqrt{C_1 u(x_k)}$. We choose $C_1 = 5$ and $C_2 = 5 C_1$. At the end of the algorithm, $C_3 \|u\|_{L^1_\rho(X)}$ samples are drawn, with $C_3 = 10$ unless stated otherwise. By varying $C_1$, $C_2$ and $C_3$, one balances efficiency and reliability. Second, we only take the first and the last two iterations into account for the update of $u$, i.e., 
\begin{align*}
    u_{(i)}(x) = \min \Bigg( \|k_n^\epsilon\|_{L^\infty(X)}, \; \norm{\left(R_{(i-2)}^* \right)^{-1} \phi(x)}_2^2, \; \norm{\left(R_{(i-1)}^* \right)^{-1} \phi(x)}_2^2 \Bigg).
\end{align*}
With these changes, the computational complexity reduces to $\mathcal{O}\left(n^3 \log\left(\| k_n^\epsilon\|_{L^\infty(X)}\right) \right)$. For more details on the implementation, we refer to the GitHub repository~\cite{github_repo}. 

We consider a couple of examples, which aim to showcase the versatility, robustness and efficiency of the algorithm. As a simple benchmark, we illustrate the improvement of RCS sampling compared to uniform sampling in Figure~\ref{fig:convergence}. Moreover, we compare the performance and efficiency of the RCS algorithm to the dense grid method~\cite{dolbeault2022optimal,adcock2020near,migliorati2021multivariate} in Figure~\ref{fig:timing}. \vspace{2mm}

\begin{enumerate}[itemsep=1em, leftmargin=*, label=\textbf{\alph*.}]
    \item \textbf{Polynomials + weighted polynomials.} 
    We consider
    \begin{equation*}
        \Phi_n = \{ p_i(x) \}_{i=0}^{n_1-1} \cup \{ w(x) p_i(x) \}_{i=0}^{n_2-1},
    \end{equation*}
    where $n = n_1 + n_2$ and $p_i$ is a polynomial basis function of degree $i$. Such bases are commonly used to approximate functions with known singular or oscillatory behaviour, which is incorporated in the weight function $w(x)$, see~\cite[Example 3]{adcock2019frames} and references therein. The function $u$ computed by the RCS algorithm for $n = 40$, $X = [-1,1]$ and $w(x) = \sqrt{x+1}$ is shown in Figure~\ref{fig:christoffelfun}(a). In Figure~\ref{fig:convergence}(a), we consider the convergence of the discrete least squares approximation of
    \begin{equation*}
            f(x) = \frac{\sqrt{x+1}}{1+5x^2} + \cos(5x) \qquad \text{on } X = [-1,1]
    \end{equation*}
    in the span of $\Phi_n$ with the same weight function.

    \item \textbf{Rational approximation with preassigned poles.} The lightning scheme uses rational functions with preassigned poles to approximate functions with singularities at known locations, a setting which often occurs when solving PDEs on domains with corners~\cite{gopal2019solving}. As a simple model problem, we consider approximation on $[0,1]$ of functions with an endpoint singularity at $x = 0$ using
    \begin{equation*}
        \Phi_n = \left\{ \frac{-q_i}{x - q_i} \right\}_{i=1}^{n_1} \cup \{p_i(x)\}_{i=0}^{n_2-1} \qquad \text{with } q_i = -\exp\left(4\left(\sqrt{i} - \sqrt{n_1}\right)\right),
    \end{equation*}
    where $p_i$ is again a polynomial basis function of degree $i$, $n = n_1 + n_2$ and $n_2 = \mathcal{O}(\sqrt{n_1})$~\cite{herremans2023resolution}. Hence, we combine partial fractions with simple poles $q_i$, which cluster towards zero along the negative real axis in a tapered fashion~\cite{trefethen2021exponential}, with a polynomial basis. Figure~\ref{fig:christoffelfun}(b.1) shows the function $u$ computed by the RCS algorithm for different values of $n$. In Figure~\ref{fig:convergence}(b.1), we compare the error of discrete least squares approximation to $f(x) = \sqrt{x}$ using $n_2 = 2\sqrt{n_1}$ for RCS samples and uniformly random samples. Similarly, Figure~\ref{fig:timing} presents a comparison with the discrete grid method, showing that the RCS algorithm (using $C_3 = 15$) achieves higher accuracy with significantly less computation time.

    The lightning approximation scheme can be extended to multiple dimensions~\cite{boulle2024multivariate}. As a multivariate example, we consider the setting in~\cite[\S4.2]{boulle2024multivariate}, namely:
    \begin{equation} \label{eq:lightning2D}
        \Phi_n = \bigcup_{j=1}^{n_1} \left\{ \frac{q_j p_k(x) p_l(y)}{Q(x,y) - q_j} \right\}_{(k,l) = (0,0)}^{(n_2-1,n_2-1)}  \cup \{ p_k(x) p_l(x) \}_{(k,l) = (0,0)}^{(n_3-1,n_3-1)}
    \end{equation}
    on $X = [-2,2]^2$, where $p_i$ is again a polynomial basis function of degree $i$,
    \[
        Q(x,y) = x^3 - 2x + 1 - y^2 \qquad \text{and} \qquad q_j = \pm i \exp(-4(\sqrt{n_1} - \sqrt{j})), \; 1 \leq j \leq n_1.
    \]
    Hence, the number of basis functions equals $n = 2n_1n_2^2 + n_3^2$. This basis is tailored to approximate functions with a curve of singularities along the elliptic curve $Q(x,y) = 0$. In Figures~\ref{fig:christoffelfun}(b.2) and~\ref{fig:samples}(b.2), the output of the RCS algorithm is shown for $n_1 = 15$, $n_2 = 3$ and $n_1 = 10$ ($n = 370$). The function $u$ is peaked along the elliptic curve. In Figure~\ref{fig:convergence}(b.2), we show the convergence behaviour of discrete least squares approximation to $f(x,y) = \lvert Q(x,y) \rvert = \lvert x^3 - 2x + 1 - y^2 \rvert$.
    
    \item \textbf{Fourier extension of a curve.} Non-orthogonal bases typically appear when approximating on irregular domains. Many methods, often implicitly, construct an orthonormal basis on a surrounding regular domain, which then loses orthogonality when restricted to the irregular one~\cite{adcock2020approximating}. For a Fourier basis, this is known as Fourier extension~\cite{matthysen2018function}. We consider the extreme case of approximating a function on a curve using a 2D Fourier basis. More specifically, we use
    \begin{equation} \label{eq:fourier}
        \Phi_n = \{\exp(2\pi i (k_x x + k_y y)) \}_{(k_x, k_y) = (-n_1, -n_1)}^{(n_1, n_1)}
    \end{equation}
    to approximate functions on the edge of a triangle in the torus $[-1/2,1/2)^2$. The $n = (2n_1+1)^2$ functions are heavily redundant on the curve, meaning that the numerical dimension $n^\epsilon \ll n$. In our algorithm, we find that $N = 10n_1+50 = \mathcal{O}(\sqrt{n})$ works well as an upper bound for $n^\epsilon$. The function $u$ computed by the RCS algorithm is shown in Figure~\ref{fig:christoffelfun}(c) for $n_1 = 15$. In Figure~\ref{fig:convergence}(c), we show the error of a discrete least squares fit to $f(x,y) = \sqrt{x^2 + y^2}$. Note that Christoffel sampling does not seem to outperform uniform sampling, a phenomenon which has been analysed for 1D Fourier extensions~\cite{herremans2025sampling,adcock2023fast}. 
    
    \item \textbf{Spectral method for fractional Laplacians.} 
    In~\cite{papadopoulos2023frame}, a spectral method is developed for fractional differential equations. More specifically, an approximation $\hat{v}$ to the solution $v$—that decays as $\lvert x\rvert \to +\infty$—of an equation $\mathcal{L}[v] = f$ on \(\mathbb{R}^d\) involving the fractional Laplacian $(-\Delta)^s$, $s \in (0,1)$, is sought in the span of extended and weighted orthogonal polynomials $\{\phi_i\}_i$,
    \[
        v \approx \hat{v} = \sum_i c_i \phi_i, \qquad
        \mathcal{L}\hat{v} = \sum_i c_i\, \mathcal{L}\phi_i \approx f.
    \]
    This equation can be discretized by sampling, a procedure known as \textit{collocation}. Hence, finding effective collocation points is equivalent to identifying good sampling points for approximation in the span of $\{\mathcal{L}\phi_i\}_i$, which can be achieved using the RCS algorithm.
    To illustrate this, we consider the two-dimensional setting described in~\cite[\S7.4]{papadopoulos2023frame}, restricting the domain to the disk of radius~10 as proposed in that example. For further details, we refer to~\cite[\S7.4]{papadopoulos2023frame}. Figures~\ref{fig:christoffelfun}(d) and~\ref{fig:samples}(d) show the function $u$ and the sample points computed by the RCS algorithm, while Figure~\ref{fig:convergence}(d) presents the convergence results comparing uniform samples, RCS samples (using $C_3 = 5$) and the deterministic samples proposed in the original paper. Note that the error plot reports $\|v - \hat{v}\|_\infty$ on the disk of radius~10, which differs from $\|f - \mathcal{L}\hat{v}\|_\infty$. We find that RCS gives the same level of accuracy, while using far fewer samples than the deterministic procedure.
\end{enumerate}

\begin{figure}
    \begin{subfigure}{\linewidth}
        \centering
        \includegraphics[width=.44\linewidth]{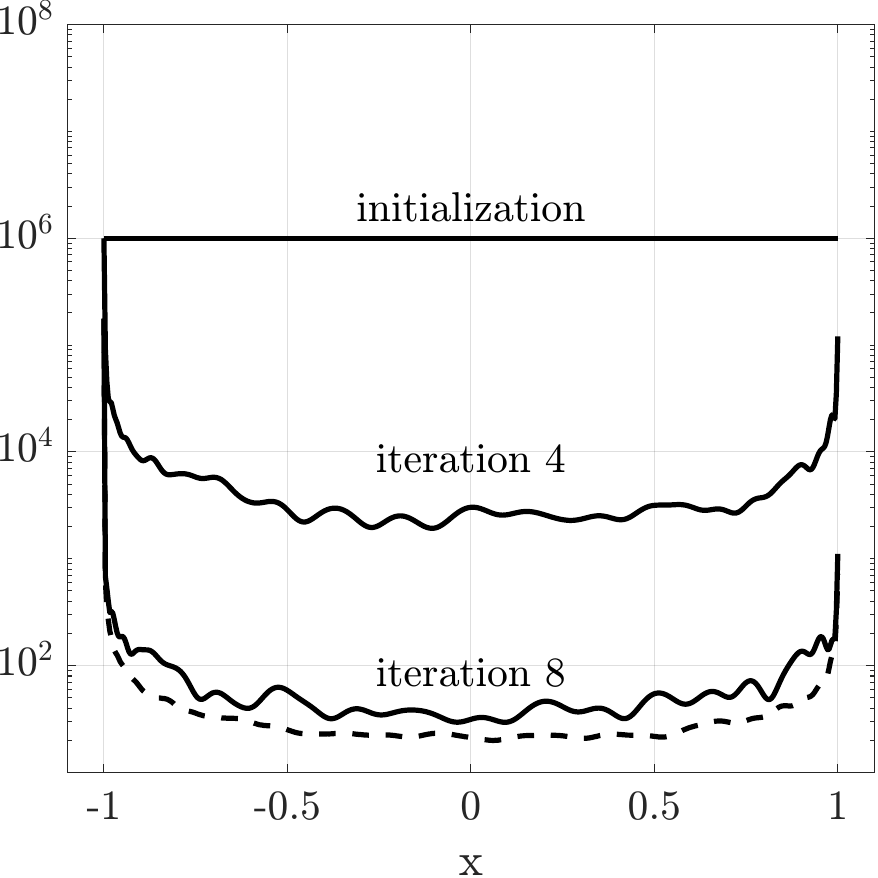}\hspace{9mm}
        \includegraphics[width=.44\linewidth]{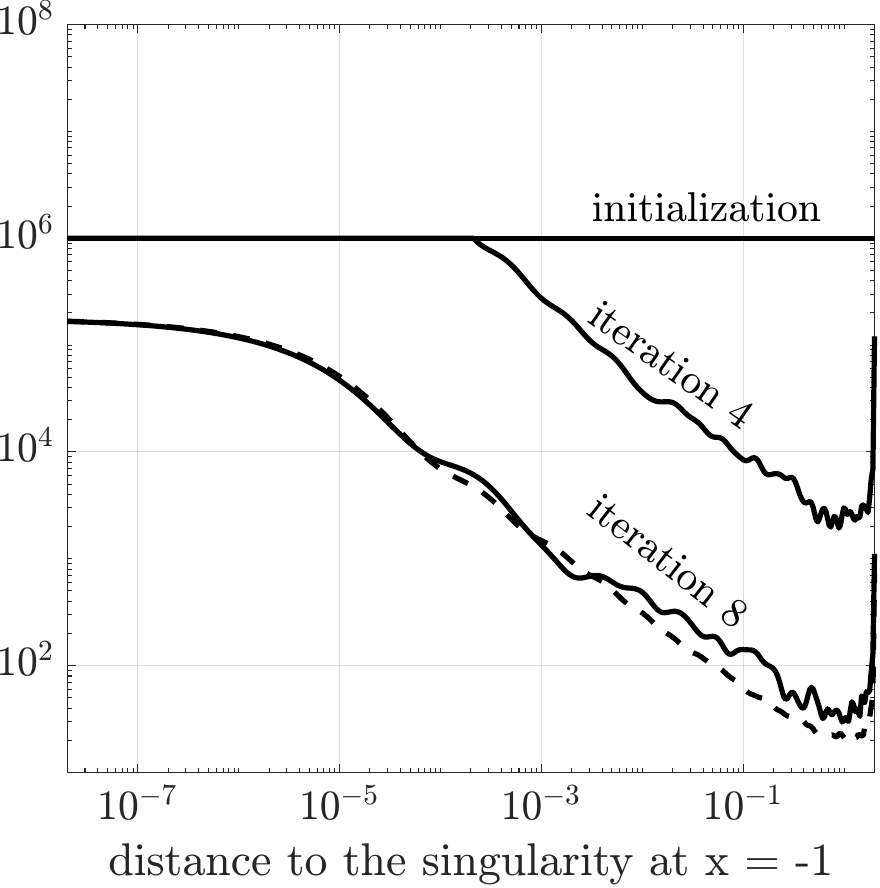}
        \caption{polynomials + weighted polynomials}
    \end{subfigure}
    \begin{subfigure}{.49\linewidth}
        \centering
        \includegraphics[width=.9\linewidth]{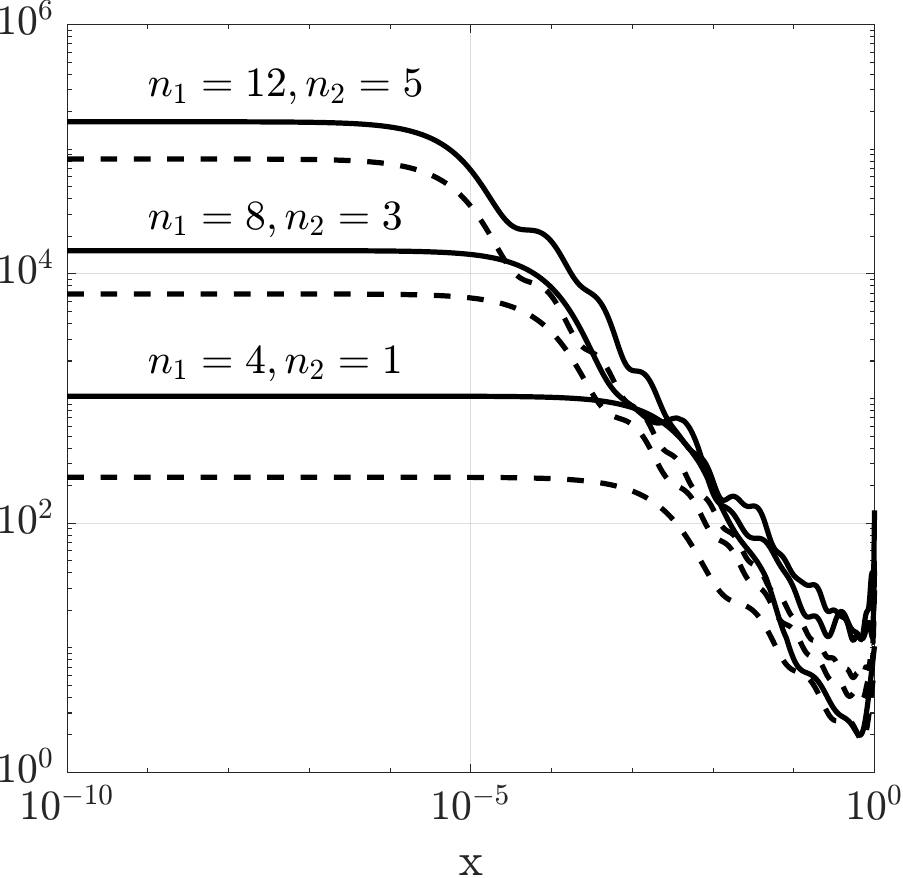}
        \caption*{(b.1) lightning approximation}
    \end{subfigure}\hfill
    \begin{subfigure}{.49\linewidth}
        \centering
        \includegraphics[width=.9\linewidth]{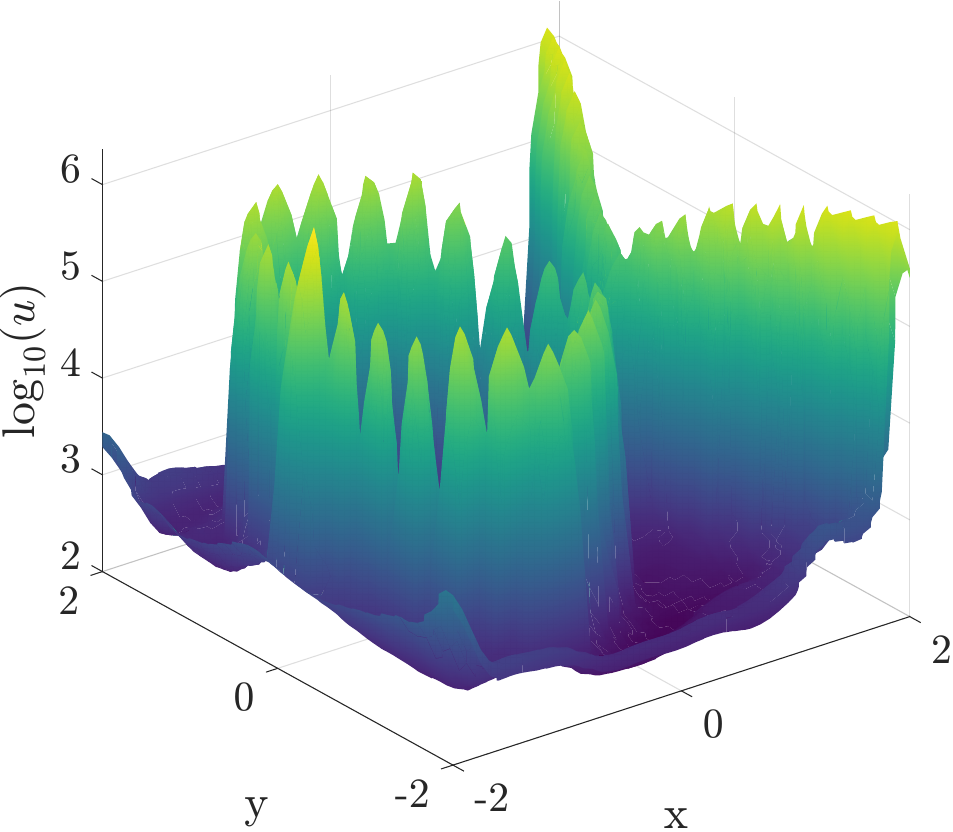}
        \caption*{(b.2) 2D lightning approximation}
    \end{subfigure}
    \begin{subfigure}{.49\linewidth}
        \centering
        \includegraphics[width=.9\linewidth]{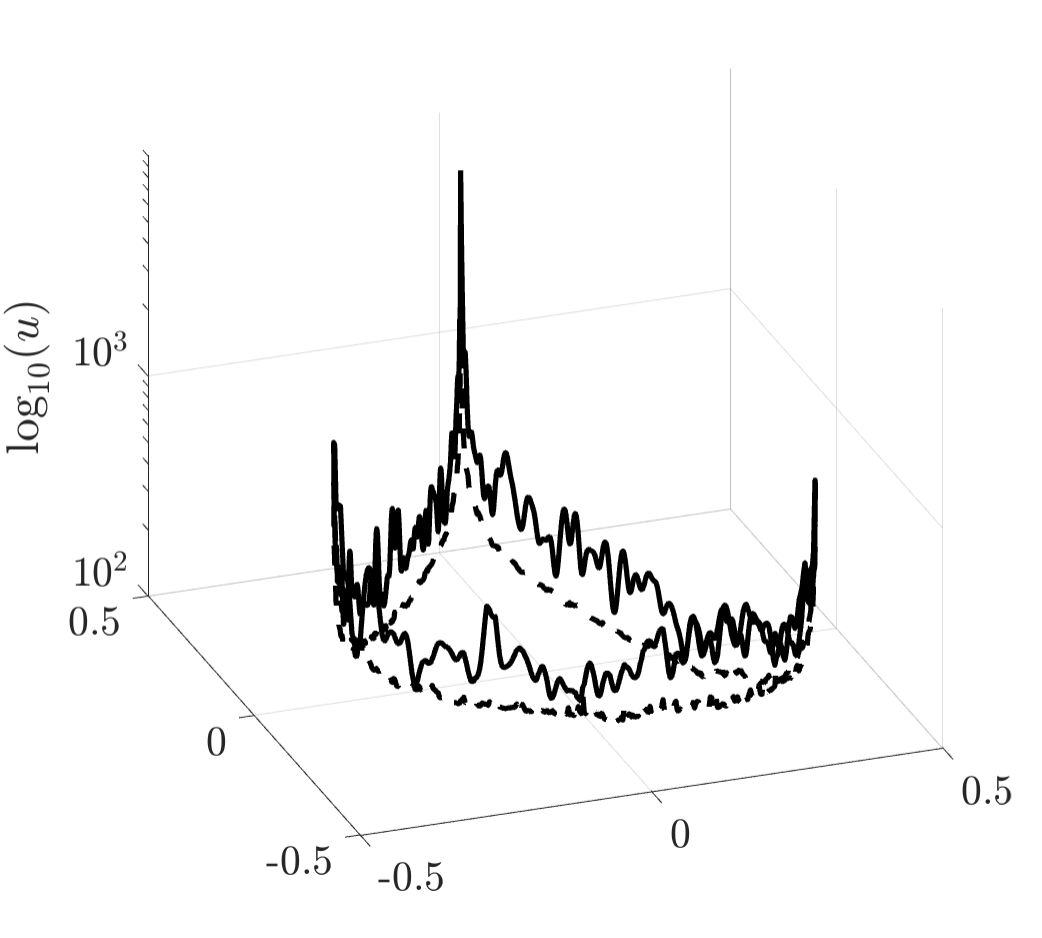}
        \caption*{(c) Fourier extension}
    \end{subfigure}\hfill
    \begin{subfigure}{.49\linewidth}
        \centering
        \includegraphics[width=.9\linewidth]{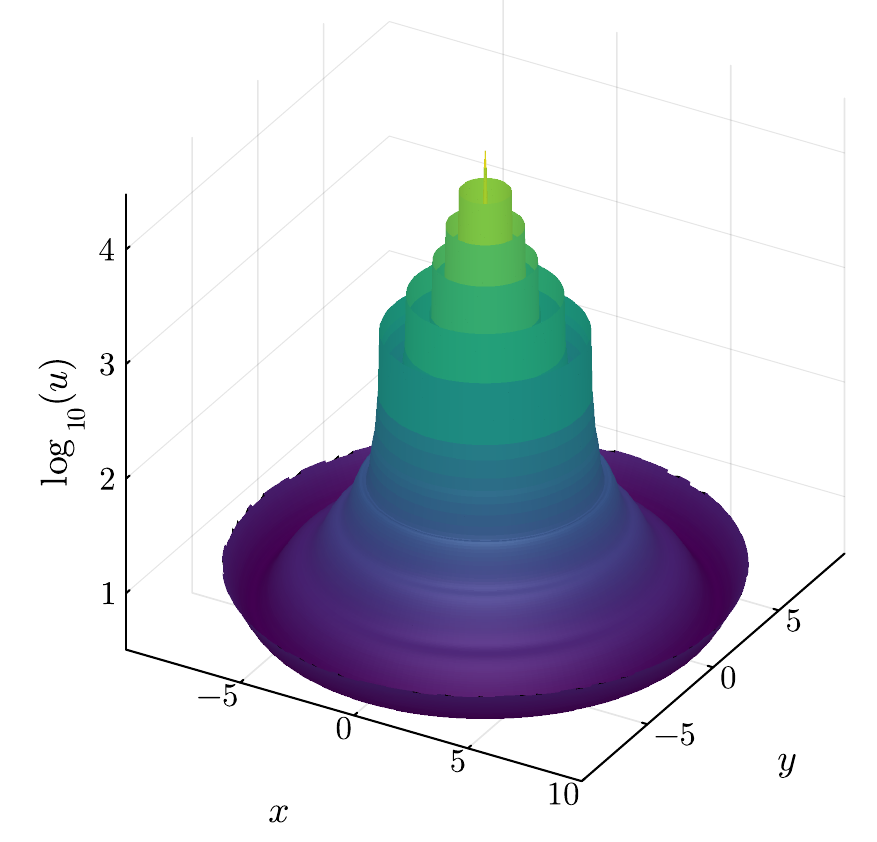}
        \caption*{(d) fractional spectral method}
    \end{subfigure}
    \caption{Illustration of the function $u$ computed by the RCS algorithm. Optionally, a comparison is made with an approximation to $k_n^\epsilon$ obtained via the dense grid method~\cite{dolbeault2022optimal,adcock2020near,migliorati2021multivariate} (dashed line) using a sufficiently dense grid. For more details, see the subsections in \S\ref{sec:numericalexperiments} corresponding to each plot label.}
    \label{fig:christoffelfun}
\end{figure}

\begin{figure}
    \begin{subfigure}{.49\linewidth}
        \centering
        \includegraphics[width=.9\linewidth]{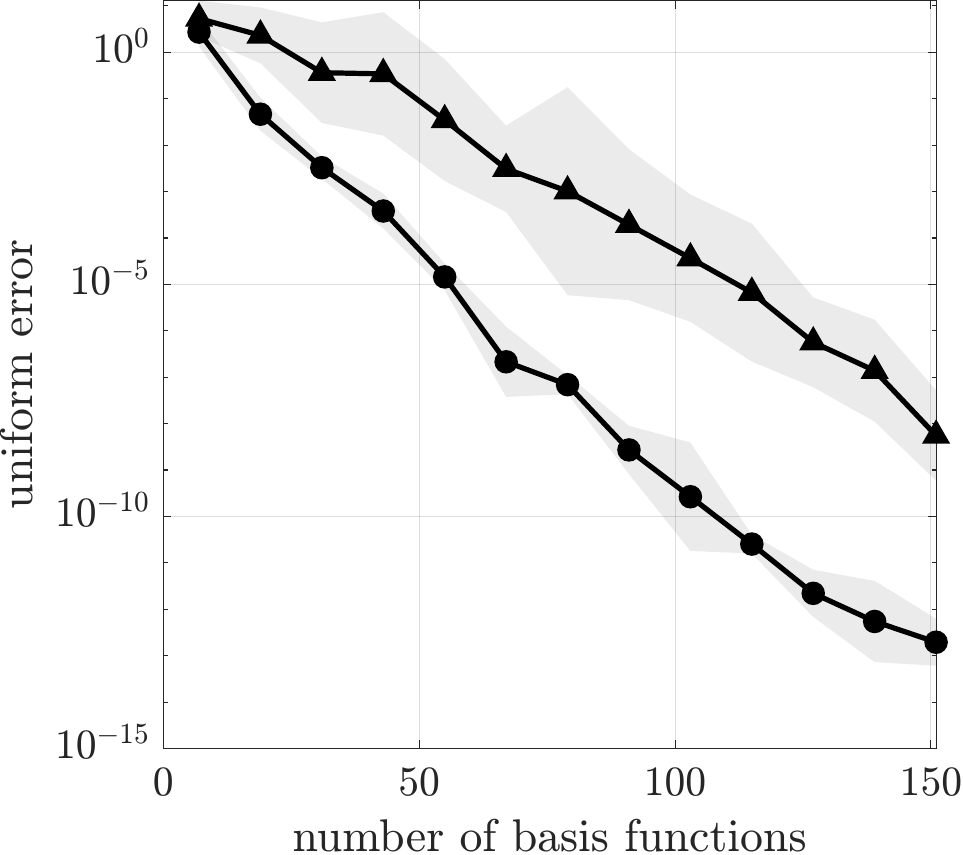}
        \caption{polynomials + weighted polynomials}
    \end{subfigure}\hfill
    \begin{subfigure}{.49\linewidth}
        \centering
        \includegraphics[width=.9\linewidth]{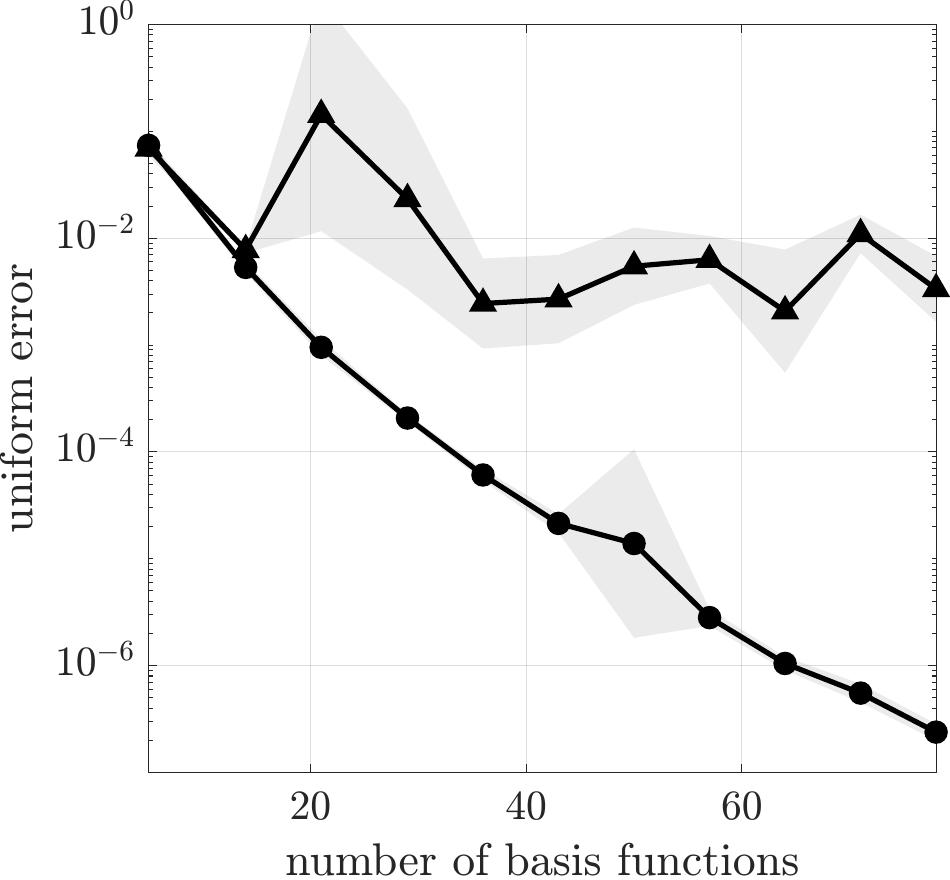}
        \caption*{(b.1) lightning approximation}
    \end{subfigure}
    \begin{subfigure}{.49\linewidth}
        \centering
        \includegraphics[width=.9\linewidth]{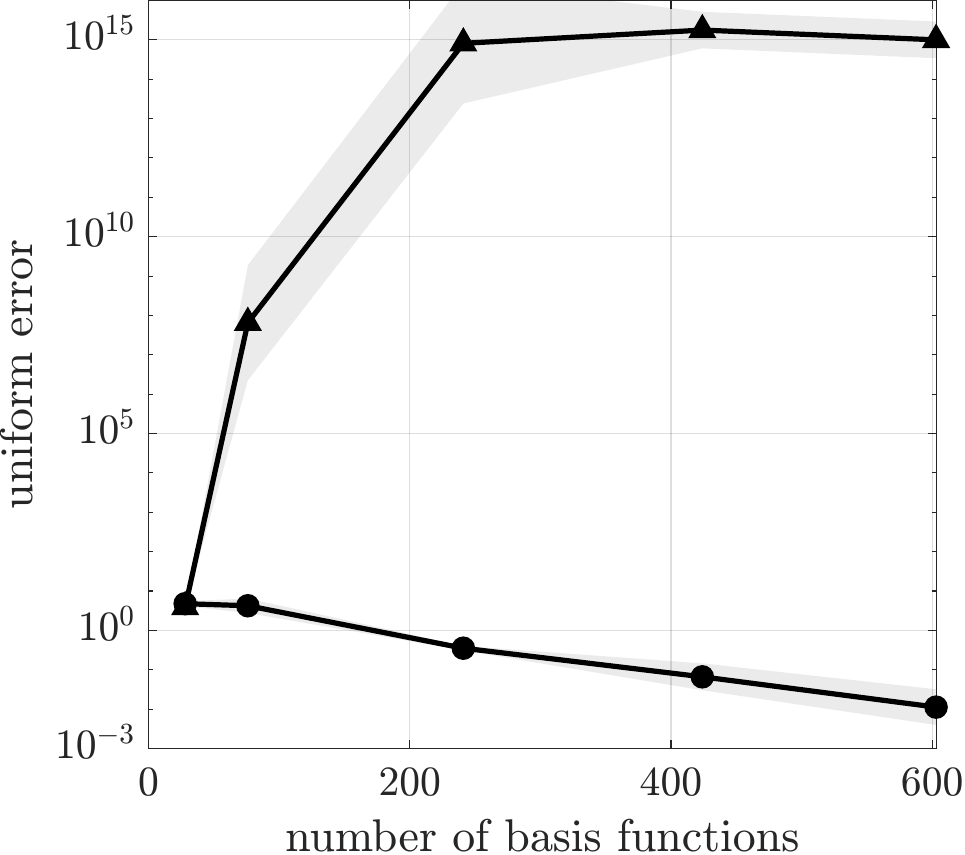}
        \caption*{(b.2) 2D lightning approximation}
    \end{subfigure}\hfill
    \begin{subfigure}{.49\linewidth}
        \centering
        \includegraphics[width=.9\linewidth]{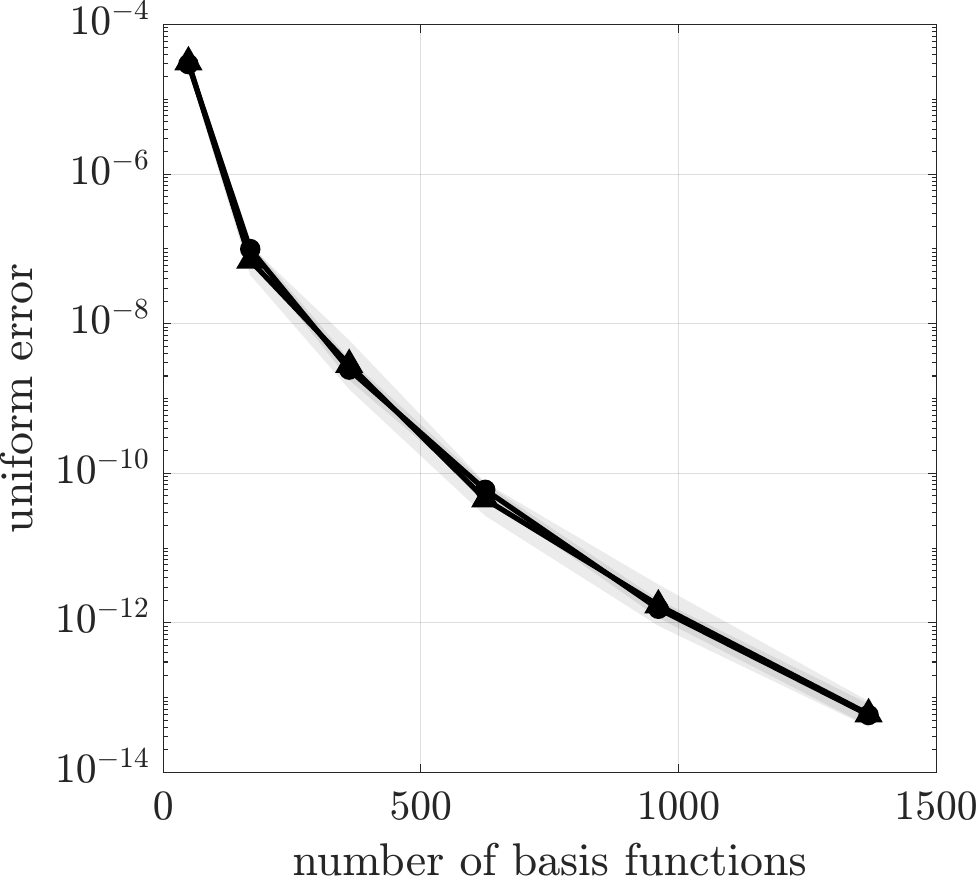}
        \caption*{(c) Fourier extension}
    \end{subfigure}
    \begin{subfigure}{\linewidth}
        \centering
        \includegraphics[width=.44\linewidth]{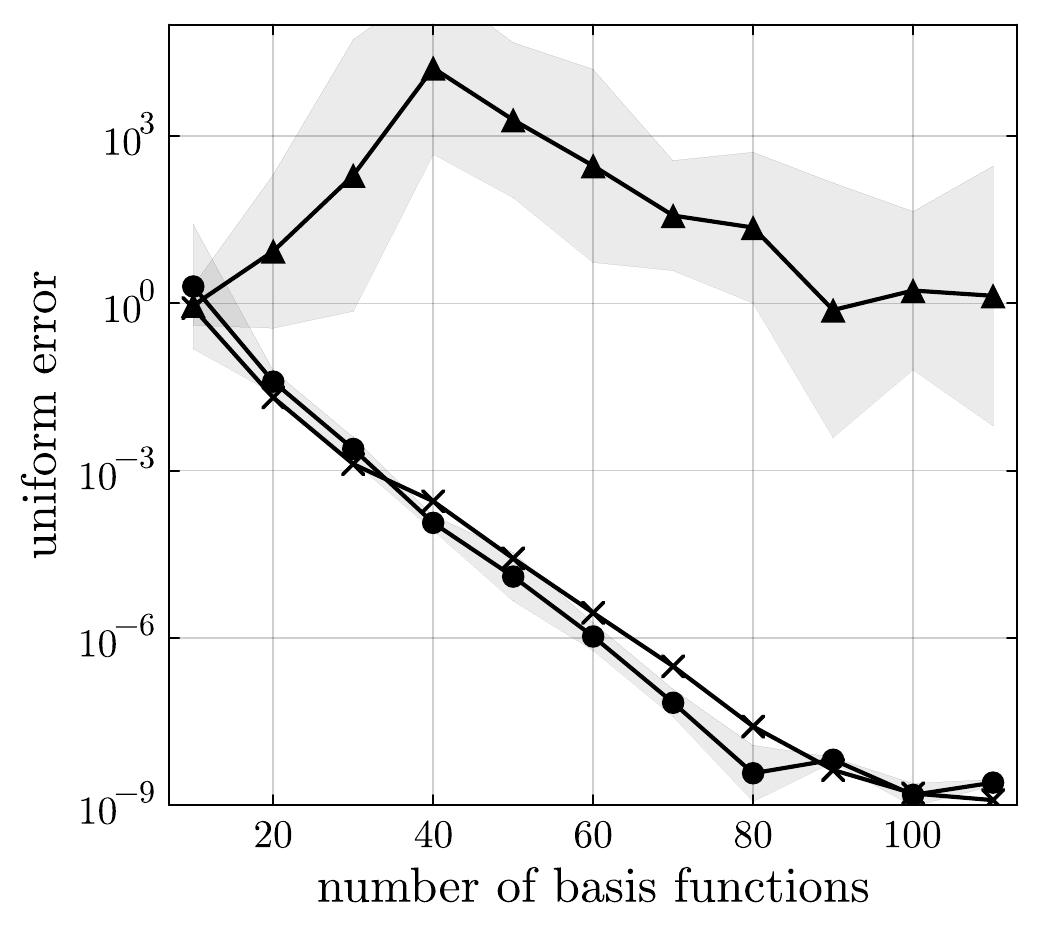}\hspace{9mm}
        \includegraphics[width=.44\linewidth]{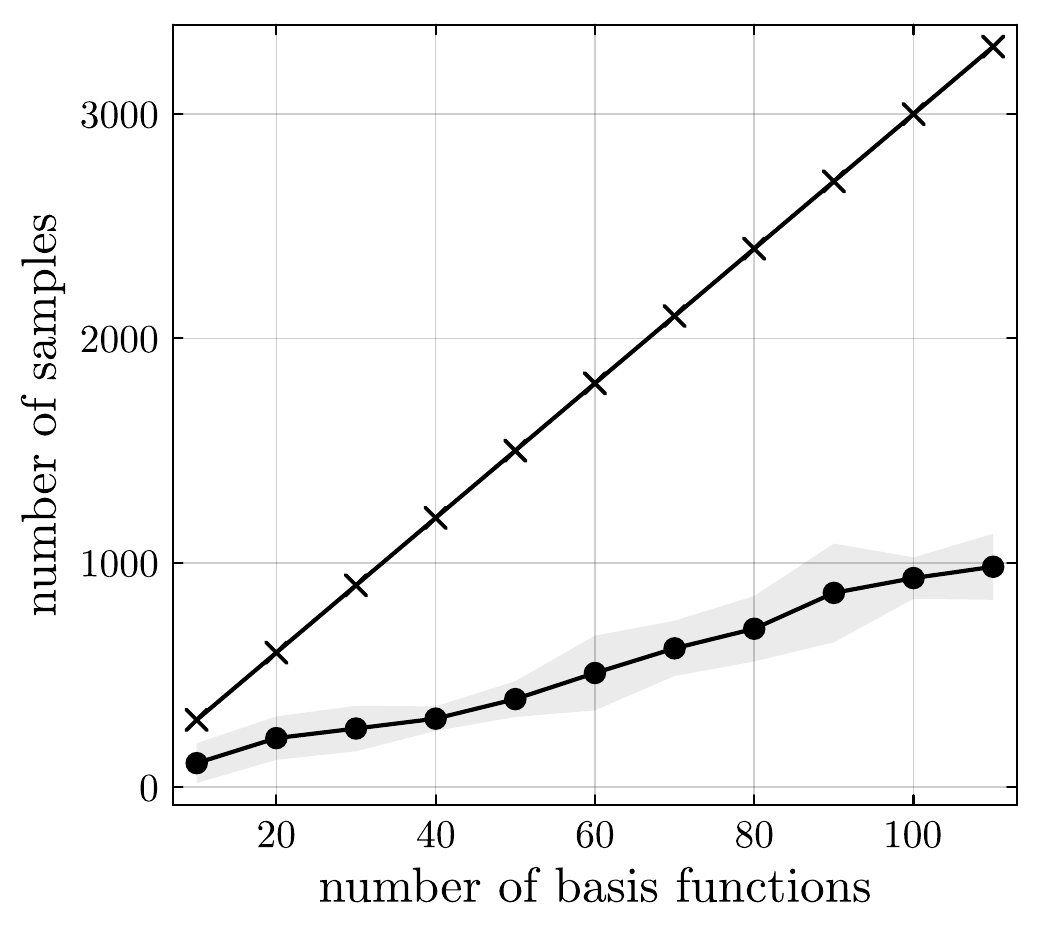}
        \caption*{(d) fractional spectral method}
    \end{subfigure}
    \caption{Convergence of discrete least squares approximation using samples computed by the RCS algorithm (circles) and the same number of uniformly random samples (triangles). The line indicates the geometric mean and the shaded area shows the variance over ten repetitions. The reported error is computed on an independent, dense grid. In (d), the deterministic sample strategy proposed in~\cite[\S7.4]{papadopoulos2023frame} is added to the comparison (crosses), including a comparison of the number of sample points. For more details, see the subsections in \S\ref{sec:numericalexperiments} corresponding to each plot label.}
    \label{fig:convergence}
\end{figure}

\begin{figure}
    \begin{subfigure}{.49\linewidth}
        \centering
        \includegraphics[width=.95\linewidth]{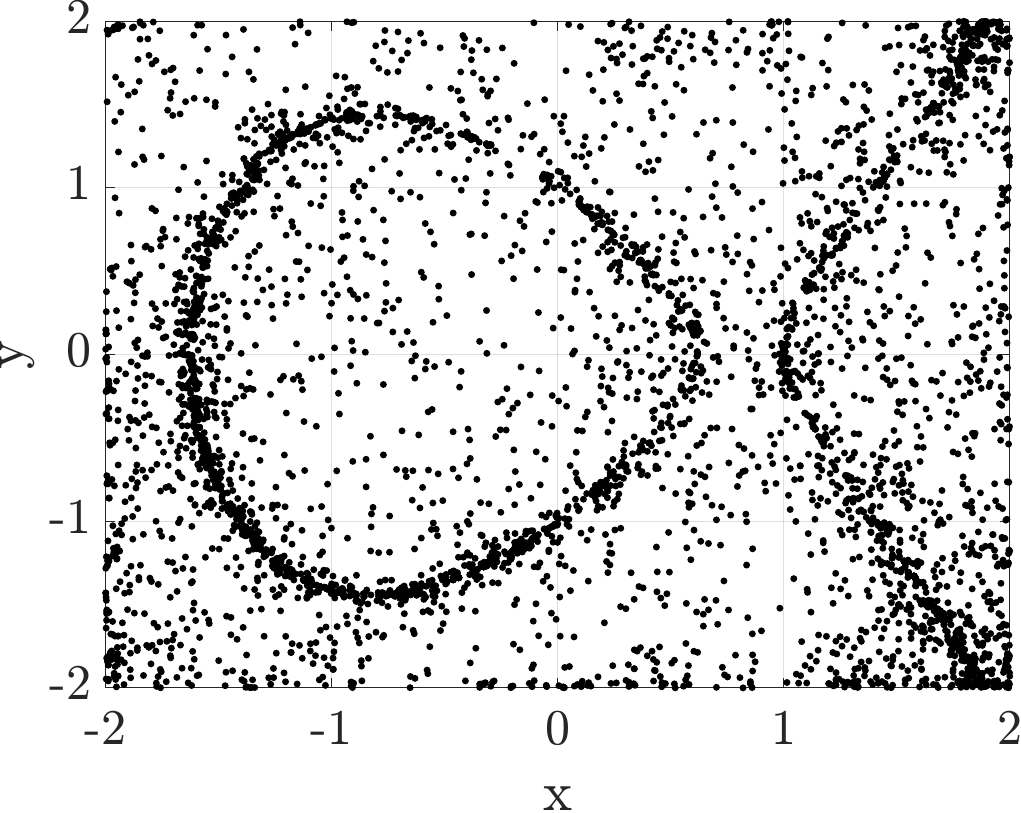}
        \caption*{(b.2) 2D lightning approximation}
    \end{subfigure}\hfill
    \begin{subfigure}{.49\linewidth}
        \centering
        \includegraphics[width=.95\linewidth]{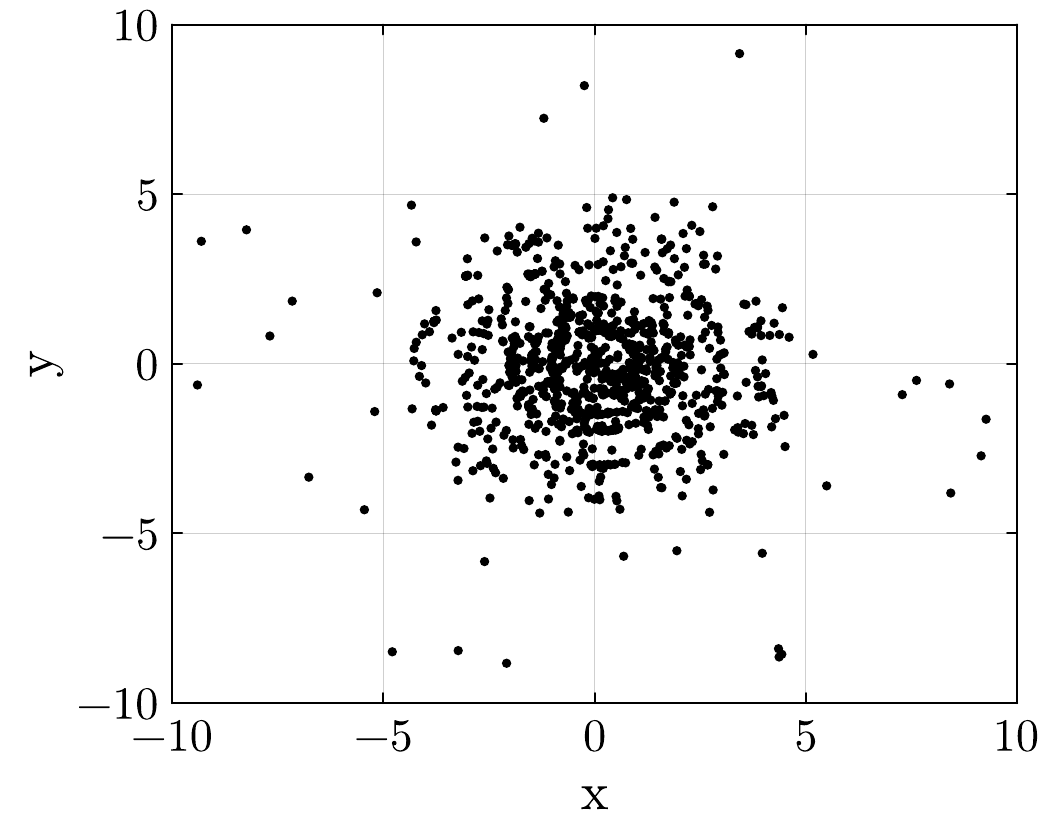}
        \caption*{(d) fractional spectral method}
    \end{subfigure}
    \caption{Illustration of the sample points outputted by the RCS algorithm. For more details, see the subsections in \S\ref{sec:numericalexperiments} corresponding to each plot label.}
    \label{fig:samples}
\end{figure}

\begin{figure}
    \centering
    \includegraphics[width=0.49\linewidth]{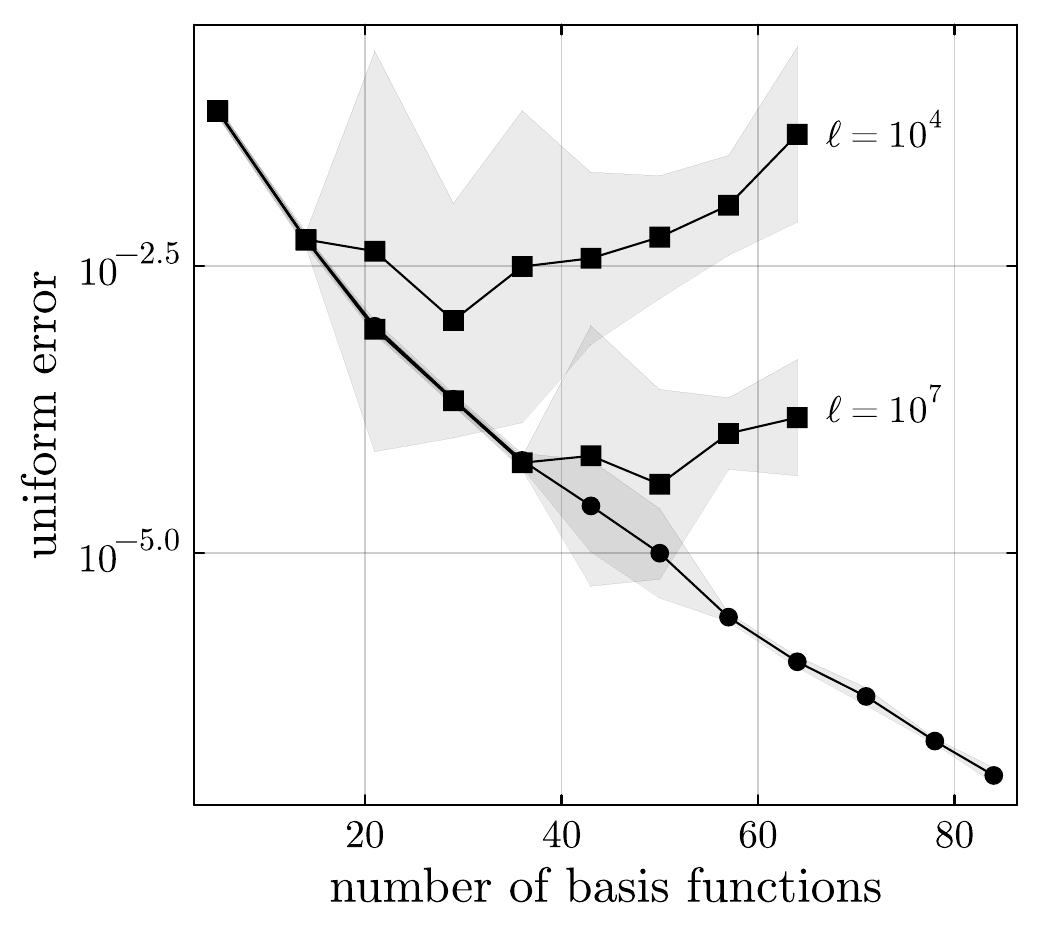}\hfill
    \includegraphics[width=0.49\linewidth]{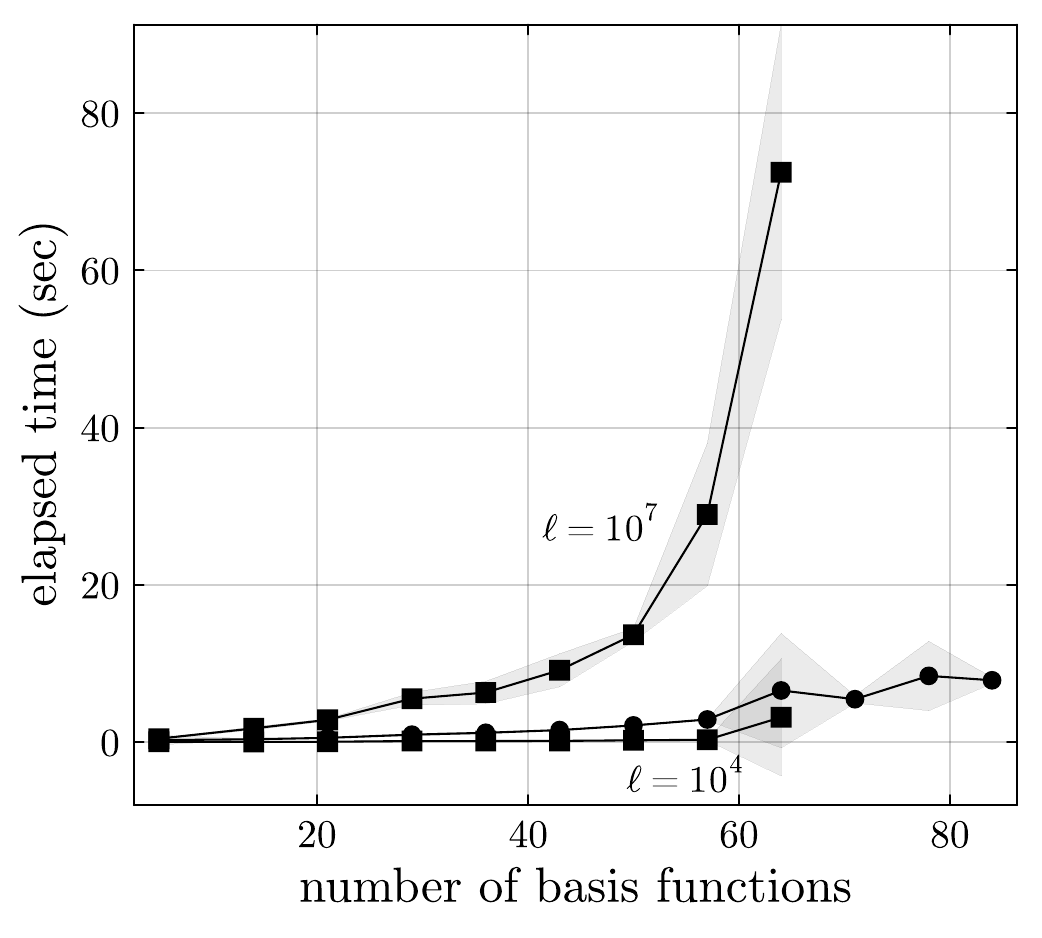}
    \caption{Comparison between the RCS algorithm (using $C_3 = 15$) and the discrete grid method~\cite{dolbeault2022optimal,adcock2020near,migliorati2021multivariate}, which approximates $k^\epsilon_n$ by discrete orthogonalization of the basis on a uniformly random grid of $\ell$ points, for the 1D lightning approximation. Left: error evaluated on an independent, dense grid; right: elapsed time required to compute sample points. The experiment was conducted on a contemporary laptop using a \texttt{Julia} implementation~\cite{github_repo}.}
    \label{fig:timing}
\end{figure}

\section{Conclusions} \label{sec:conclusion}
The refinement-based Christoffel sampling (RCS) algorithm provides an efficient and robust approach for discretizing non-orthogonal bases, yielding stable and accurate least-squares approximations. By iteratively refining a random sampling measure, the method naturally adapts to the underlying geometry of the basis while accounting for the effects of finite precision

There remain, however, several opportunities for further improvement. In the highly redundant regime, where the numerical dimension $n^\epsilon$ is much smaller than $n$, there still seems to be room for further computational savings. In particular, the evaluation of $u$ may be accelerated via randomized numerical linear algebra, since the matrix $A$ formed at each iteration is expected to be low rank and therefore suitable for oblivious dimensionality-reduction techniques~\cite{avron2017sharper}. Developing reliable and practical estimators for the numerical dimension—see also the discussion in \S\ref{sec:discussion}—is another important step toward making the scheme effective in this setting. Finally, a more general extension would be to adapt the algorithm to unbounded domains, for which $k_n^\epsilon$ may fail to be uniformly bounded.

\section*{Acknowledgments}
The authors wish to thank Daan Huybrechs and Christopher Musco for fruitful discussions on the topic. BA was supported by the the Natural Sciences and Engineering Research Council of Canada
(NSERC) through grant RGPIN-2021-611675. AH is a PhD fellow of the Research Foundation Flanders (FWO), funded by grant 11P2T24N, and was additionally supported by an FWO travel grant for a long stay abroad (V412525N).

\bibliographystyle{siamplain}
\bibliography{references}

\end{document}